\documentclass{amsart}

\usepackage{amscd,amssymb,amsmath,amsbsy,amsthm}
\usepackage[all]{xy}
\usepackage[colorlinks,backref,urlcolor=blue]{hyperref}

\newtheorem{theorem}{Theorem}[section]
\newtheorem{lemma}[theorem]{Lemma}
\newtheorem{corollary}[theorem]{Corollary}
\newtheorem{prop}[theorem]{Proposition}

\theoremstyle{definition}
\newtheorem{definition}[theorem]{Definition}
\newtheorem{example}[theorem]{Example}
\newtheorem{remark}[theorem]{Remark}

\newcommand{\Z}{\mathbb{Z}}
\newcommand{\Q}{\mathbb{Q}}
\newcommand{\R}{\mathbb{R}}
\newcommand{\C}{\mathbb{C}}

\newcommand{\QP}{\mathbb{QP}}
\newcommand{\CP}{\mathbb{CP}}
\newcommand{\RP}{\mathbb{RP}}
\newcommand{\bP}{\mathbb{P}}

\newcommand{\sV}{\mathsf{V}}
\newcommand{\sW}{\mathsf{W}}
\newcommand{\sE}{\mathsf{E}}

\newcommand{\G}{\Gamma}

\renewcommand{\k}{\Bbbk}
\newcommand{\RR}{\mathcal{R}}
\newcommand{\VV}{\mathcal{V}}
\newcommand{\A}{{\mathcal{A}}}

\newcommand{\CC}{{\mathcal{C}}}
\newcommand{\WW}{\mathcal{W}}

\DeclareMathOperator{\gr}{gr}
\DeclareMathOperator{\im}{im}

\DeclareMathOperator{\codim}{codim}

\DeclareMathOperator{\ab}{{ab}}

\DeclareMathOperator{\supp}{supp}

\DeclareMathOperator{\Hom}{{Hom}}

\DeclareMathOperator{\FF}{{F}}
\DeclareMathOperator{\FP}{{FP}}

\DeclareMathOperator{\init}{in}
\DeclareMathOperator{\lk}{lk}
\DeclareMathOperator{\rk}{rk}

\DeclareMathOperator{\Tors}{Tors}

\DeclareMathOperator{\Grass}{Gr}

\DeclareMathOperator{\TC}{TC}

\newcommand{\m}{\mathfrak{m}}

\newcommand{\PL}{\scriptscriptstyle{\rm PL}}

\newcommand{\wX}{\widetilde{X}}
\newcommand{\wG}{\widehat{G}}

\newcommand{\surj}{\twoheadrightarrow}
\newcommand{\inj}{\hookrightarrow}
\newcommand{\isom}{\xrightarrow{\simeq}}
\newcommand{\compl}{\scriptscriptstyle{\complement}}

\newcommand{\abs}[1]{\left| #1 \right|}

\def\set#1{{\{ #1\}}}

\newcommand{\bigmid}{\:\big|  \big.\:}
\newcommand{\pt}{\{\text{pt}\}}
\newcommand{\rat}{\mathcal{R}\Gamma_{\iota}}

\title[Finiteness in free abelian covers]%
{Geometric and homological finiteness in \\
free abelian covers}

\author[Alexander~I.~Suciu]{Alexander~I.~Suciu}

\address{Department of Mathematics,
Northeastern University,
Boston, MA 02115, USA}
\urladdr{http://www.math.neu.edu/\~{}suciu}
\email{a.suciu@neu.edu}

\thanks{Partially supported by NSA grant H98230-09-1-0021
and NSF grant DMS--1010298}


\begin{document}

\begin{abstract}
We describe some of the connections between  
the Bieri--Neumann--Strebel--Renz invariants,  
the Dwyer--Fried invariants, and the cohomology 
support loci of a space $X$. Under suitable 
hypotheses, the geometric and homological finiteness 
properties of regular, free abelian covers of $X$ 
can be expressed in terms of the resonance varieties, 
extracted from the cohomology ring of $X$. 
In general, though, translated components 
in the characteristic varieties affect the answer.
We illustrate this theory in the setting of toric 
complexes, as well as smooth, complex projective 
and quasi-projective varieties, with special emphasis 
on configuration spaces of Riemann surfaces and 
complements of hyperplane arrangements. 
\end{abstract}

\maketitle
\tableofcontents
\setcounter{tocdepth}{2} 

\section{Introduction}
\label{sect:intro}

\subsection{Finiteness properties}
\label{intro:fp}

This investigation is motivated by two seminal papers that 
appeared in 1987:  one by  Bieri, Neumann, and Strebel \cite{BNS}, 
and the other by Dwyer and Fried \cite{DF}.  Both papers  
dealt with certain finiteness properties of normal 
subgroups of a group (or regular covers 
of a  space), under the assumption that 
the factor group (or the group of deck transformations)  
is free abelian. 

In \cite{BNS}, Bieri, Neumann, and Strebel associate to every 
finitely generated group $G$ a subset  $\Sigma^1(G)$ of the unit 
sphere $S(G)$ in the real vector space $\Hom(G,\R)$. This 
``geometric'' invariant of the group $G$ is cut out of the sphere 
by open cones, and is independent of a finite generating set 
for $G$.  In  \cite{BR}, Bieri and Renz introduced a nested family 
of higher-order invariants, $\{\Sigma^i(G,\Z)\}_{i\ge 1}$, which 
record the finiteness properties of 
normal subgroups of $G$ with abelian quotients. 

In a recent paper \cite{FGS}, Farber, Geoghegan and Sch\"utz 
further extended these definitions. To each connected, 
finite-type CW-complex $X$, these authors assign a 
sequence of invariants, $\{\Sigma^i(X,\Z)\}_{i\ge 1}$, 
living in the unit sphere $S(X)\subset H^1(X,\R)$. 
The sphere $S(X)$ can be thought of as parametrizing 
all free abelian covers of $X$, while the $\Sigma$-invariants 
(which are again open subsets), keep track of the geometric 
finiteness properties of those covers. 

Another tack was taken by Dwyer and Fried in \cite{DF}.  
Instead of looking at all free abelian covers of $X$ at once, 
they fix the rank, say $r$, of the deck-transformation group, 
and view the resulting covers as being parametrized by 
the rational Grassmannian $\Grass_r(H^1(X,\Q))$. 
Inside this Grassmannian, they consider the subsets 
$\Omega^i_r(X)$, consisting of all covers for which 
the Betti numbers up to degree $i$ are finite, and 
show how to determine these sets in terms of the support 
varieties of the relevant Alexander invariants of $X$.  
Unlike the $\Sigma$-invariants, though, the $\Omega$-invariants 
need not be open subsets, see \cite{DF} and \cite{Su11}.

Our purpose in this note is to explore several connections 
between the geometric and homological invariants 
of a given space $X$, and use these connections to 
derive useful information about the rather mysterious 
$\Sigma$-invariants from concrete knowledge of the more 
accessible $\Omega$-invariants.

\subsection{Characteristic varieties and $\Omega$-sets}
\label{intro:cvs omega}

Let $G=\pi_1(X,x_0)$ be the fundamental group of $X$, 
and let $\wG=\Hom(G,\C^{\times})$ be the group of 
complex-valued characters on $G$, thought of as 
the parameter space for rank~$1$ local systems on $X$. 
The key tool for comparing the aforementioned 
invariants are the {\em characteristic varieties} $\VV^i(X)$, 
consisting of those characters $\rho \in \wG$ 
for which $H_j(X,\C_{\rho})\ne 0$, for some $j\le i$. 

Let $X^{\ab}$ be the universal abelian cover of $X$, 
with group of deck transformations $G_{\ab}$. 
We may view each homology group $H_j(X^{\ab},\C)$ 
as a finitely generated module over the Noetherian ring 
$\C{G_{\ab}}$.   As shown in \cite{PS-plms}, the variety 
$\VV^i(X)$ is the support locus for the direct sum of 
these modules, up to degree $i$. 
It follows then from the work of Dwyer and Fried \cite{DF}, 
as further reinterpreted in \cite{PS-plms, Su11}, that 
$\Omega^i_r(X)$ consists of those $r$-planes $P$ 
inside $H^1(X,\Q)$ for which the algebraic torus 
$\exp(P \otimes \C)$ intersects the variety $\VV^i(X)$ in 
only finitely many points.

Let $\WW^i(X)$ be the intersection of $\VV^i(X)$ with the 
identity component of $\wG$, and 
let $\tau_1(\WW^i(X))$ be the set of points $z\in H^1(X,\C)$ 
such that $\exp(\lambda z)$ belongs to $\WW^i(X)$, for all 
$\lambda\in \C$.  As noted in \cite{DPS-duke}, this set 
is a finite union of rationally defined subspaces.   
For each $r\ge 1$, we then have 
\begin{equation}
\label{eq:intro schubert}
\Omega^i_r(X) \subseteq  \Grass_r(H^1(X,\Q)) \setminus 
 \sigma_r(\tau_1^{\Q}(\WW^i(X)) ),
\end{equation}
where $\tau_1^{\Q}$ denotes the rational points on $\tau_1$, 
and $\sigma_r(V)$ denotes the variety of incident $r$-planes 
to a homogeneous subvariety $V\subset H^1(X,\Q)$, see \cite{Su11}.  
There are many classes of spaces for which inclusion 
\eqref{eq:intro schubert} holds as equality---for instance, 
toric complexes, or, more generally, the ``straight spaces" 
studied in \cite{Su-aspm}---but, in general, the inclusion is strict.

\subsection{Comparing the $\Omega$-sets and the $\Sigma$-sets}
\label{intro:cvs os}
A similar inclusion holds for the BNSR-invariants.  As shown 
in \cite{FGS},  the set $\Sigma^i(X, \Z)$ consists of those 
elements $\chi\in S(X)$ for which the homology of $X$ with 
coefficients in the Novikov-Sikorav completion $\widehat{\Z{G}}_{-\chi}$ 
vanishes, up to degree $i$. Using this interpretation, we 
showed in \cite{PS-plms} that the following inclusion holds:
\begin{equation}
\label{eq:intro sigma}
\Sigma^i(X, \Z)\subseteq S(X) \setminus 
S(\tau^{\R}_1(\WW^i(X))),
\end{equation}
where $\tau_1^{\R}$ denotes the real points on $\tau_1$, 
and $S(V)$ denotes the intersection of $S(X)$ with a 
homogeneous subvariety $V\subset H^1(X,\R)$.  
Again, there are several classes of spaces for which inclusion 
\eqref{eq:intro sigma} holds as equality---for instance, 
nilmanifolds, or compact K\"ahler manifolds without elliptic 
pencils with multiple fibers---but, in general, the inclusion is strict.

Clearly, formulas \eqref{eq:intro schubert} and \eqref{eq:intro sigma} 
hint at a connection between the Dwyer--Fried invariants and  
the Bieri--Neumann--Strebel--Renz invariants of a space $X$. 
We establish an explicit connection here, by comparing 
the conditions insuring those inclusions hold as equalities. 
Our main results reads as follows.

\begin{theorem} 
\label{thm:intro main}
Let $X$ be a connected CW-complex with finite $k$-skeleton.  
Suppose that, for some $i\le k$, 
\[
\Sigma^i(X,\Z) = S(X)\setminus S(\tau^{\R}_1 (\WW^i(X))).
\]
Then, for all $r\ge 1$,
\[
\Omega^i_r(X) = \Grass_r(H^1(X,\Q)) \setminus 
\sigma_r ( \tau^{\Q}_1 (\WW^i(X)) ).
\]
\end{theorem}

Simple examples show that, in general, the above implication 
cannot be reversed.

The main usefulness of Theorem \ref{thm:intro main} resides in 
the fact that it allows one to show that the $\Sigma$-invariants 
of a space $X$ are smaller than the upper bound given 
by \eqref{eq:intro sigma}, once one finds certain components 
in the characteristic varieties of $X$ (e.g., positive-dimensional 
translated subtori) insuring that the $\Omega$-invariants are 
smaller than the upper bound given by \eqref{eq:intro schubert}.

\subsection{Formality, straightness, and resonance}
\label{intro:res os}

The above method can still be quite complicated, in that 
it requires computing cohomology with coefficients in 
rank $1$ local systems on $X$.  As noted in \cite{PS-plms} 
and \cite{Su-aspm}, though, in favorable situations 
the right-hand sides of  \eqref{eq:intro schubert} 
and \eqref{eq:intro sigma} can be expressed in 
terms of ordinary cohomological data.  

By definition, the $i$-th {\em resonance variety}\/ of $X$, 
with coefficients in a field $\k$ of characteristic $0$, is the set 
$\RR^i(X,\k)$ of elements $a\in H^1(X,\k)$ for which the 
cochain complex whose terms are the cohomology groups 
$H^j(X,\k)$, and whose differentials are given by multiplication by 
$a$ fails to be exact in some degree $j\le i$.  It is readily 
seen that each of these sets is a homogeneous 
subvariety of $H^1(X,\k)$.  

If $X$ is $1$-formal (in the sense of rational homotopy theory), 
or, more generally, if $X$ is locally $1$-straight, then 
$\tau_1(\WW^1(X))=\TC_1(\WW^1(X))=\RR^1(X)$.  
Thus, formulas \eqref{eq:intro schubert} and 
\eqref{eq:intro sigma} yield the following inclusions:
\begin{align}
\label{eq:intro omegares}
\Omega^1_r(X) &\subseteq  \Grass_r(H^1(X,\Q)) \setminus 
\sigma_r (  \RR^1(X,\Q) ),
\\
\label{eq:intro sigmares}
\Sigma^1(X, \Z) & \subseteq S(X) \setminus 
S( \RR^1(X,\R)).
\end{align}

If $X$ is locally $k$-straight, then the analogue of 
\eqref{eq:intro omegares} holds in degrees $i\le k$, 
with equality if $X$ is $k$-straight.  In general, though, 
neither \eqref{eq:intro omegares} nor 
\eqref{eq:intro sigmares} is an equality.

Applying now Theorem \ref{thm:intro main}, we obtain the 
following corollary.

\begin{corollary} 
\label{cor:bnsrv intro}
Let $X$ be a locally $1$-straight space (for instance, a 
$1$-formal space). Suppose 
$\Sigma^1(X,\Z) = S(X)\setminus S(\RR^1(X,\R))$.  
Then, for all $r\ge 1$,
\[  
\Omega^1_r(X) = \Grass_r(H^1(X,\Q)) 
\setminus  \sigma_r (  \RR^1(X,\Q) ).
\] 
\end{corollary}

\subsection{Applications}
\label{intro:apps}

We illustrate the theory outlined above with several classes 
of examples, coming from toric topology and algebraic geometry, 
as well as the study of configuration spaces and hyperplane 
arrangements. 

\subsubsection{Toric complexes}
Every simplicial complex $L$ on $n$ vertices determines 
a subcomplex $T_L$ of the $n$-torus, with  
$k$-cells corresponding to the $(k-1)$-simplices of $L$.   
The fundamental group $\pi_1(T_L)$ is the
right-angled Artin group $G_{L}$ 
attached to the $1$-skeleton of $L$, 
while a classifying space for $G_{L}$ is the toric 
complex associated to the flag complex $\Delta_{L}$. 

It is known that all toric complexes are both straight 
and formal; their characteristic and resonance varieties 
 were computed in \cite{PS-adv},  
whereas the $\Sigma$-invariants of right-angled Artin 
groups were computed in \cite{MMV, BG}.  
These computations, as well as work 
from \cite{PS-plms, Su-aspm} show that 
$\Omega^1_r(T_L) =  \sigma_r (  \RR^i(T_L,\Q) )^{\compl}$ and 
$\Sigma^i(G_L,\Z) \subseteq S(\RR^i(G_L,\R))^{\compl}$, 
though this last inclusion may be strict, unless 
a certain torsion-freeness assumption on the subcomplexes 
of $\Delta_L$ is satisfied. 

\subsubsection{Quasi-projective varieties}
The basic structure of the characteristic varieties of smooth, 
complex quasi-projective varieties was determined by 
Arapura \cite{Ar}, building on work of Beauville, 
Green and Lazarsfeld, and others.  In particular, 
if $X$ is such a variety, then all the components 
of $\WW^1(X)$ are torsion-translated subtori. 

If $X$ is also $1$-formal (e.g., if $X$ is compact), 
then inclusions \eqref{eq:intro omegares} and 
\eqref{eq:intro sigmares} hold, but not always as equalities. 
For instance, if $\WW^1(X)$ has a $1$-dimensional 
component not passing through $1$, and $\RR^1(X,\C)$ has 
no codimension-$1$ components, then, as shown in \cite{Su-aspm},  
inclusion \eqref{eq:intro omegares} is strict for $r=2$, 
and thus, by Corollary \ref{cor:bnsrv intro}, 
inclusion \eqref{eq:intro sigmares} is also strict. 

\subsubsection{Configuration spaces}
An interesting class of quasi-projective varieties 
is provided by the configuration spaces $X=F(\Sigma_g,n)$ 
of $n$ ordered points on a closed Riemann surface of genus $g$.  
If $g=1$ and $n\ge 3$, then the resonance variety $\RR^1(X,\C)$ 
is irreducible and non-linear, and so $X$ is not $1$-formal, 
by \cite{DPS-duke}.    To illustrate the computation of the 
$\Omega$-sets and $\Sigma$-sets in such a non-formal 
setting, we work out the details for $F(\Sigma_1,3)$. 

\subsubsection{Hyperplane arrangements} 
Given a finite collection of hyperplanes, $\A$, in a 
complex vector space $\C^{\ell}$, the complement 
$X(\A)$ is a smooth, quasi-projective variety; moreover, 
$X(\A)$ is formal, locally straight, but not always straight.
Arrangements of hyperplanes have been the main driving 
force behind the development of the theory of cohomology 
jump loci, and still provide a rich source of motivational 
examples for this theory.   Much is known about the characteristic 
and resonance varieties of arrangement complements; 
in particular, $\RR^1(X(\A),\C)$ admits a purely combinatorial 
description, owing to the pioneering work of Falk \cite{Fa97}, 
as sharpened by many others since then. We give here 
both lower bounds and upper bounds for the $\Omega$-invariants  
and the $\Sigma$-invariants of arrangements. 

In \cite{Su-aspm}, we gave an example of an  
arrangement $\A$ for which the Dwyer--Fried set 
$\Omega^1_2(X(\A))$ is strictly contained in 
$\sigma_2 (  \RR^1(X(\A),\Q) )^{\compl}$. 
Using Corollary \ref{cor:bnsrv intro}, we show here that 
the BNS invariant $\Sigma^1(X(\A),\Z)$ is strictly contained in 
$S(\RR^1(X(\A),\R))^{\compl}$.  This answers a question first 
raised at an Oberwolfach Mini-Workshop \cite{FSY}, and revisited  
in \cite{PS-plms, Su-conm}. 

\section{Characteristic and resonance varieties}
\label{sec:cvs}

We start with a brief review of the characteristic 
varieties of a space, and their relation to the 
resonance varieties, via two kinds of tangent 
cone constructions.

\subsection{Jump loci for twisted homology}
\label{subsec:jumps} 
Let $X$ be a connected CW-complex with finite $k$-skeleton, 
for some $k\ge 1$.  Without loss of generality, we may assume $X$ 
has a single $0$-cell, call it $x_0$.  Let $G=\pi_1(X,x_0)$ 
be the fundamental group of $X$, and let $\wG=\Hom(G,\C^{\times})$ 
be the group of complex characters of $G$.  Clearly, 
$\wG=\wG_{\ab}$, where $G_{\ab}=H_1(X,\Z)$ is the 
abelianization of $G$.  The universal coefficient 
theorem allows us to identify $\wG = H^1(X,\C^{\times})$. 

Each character $\rho\colon G\to \C^{\times}$ determines 
a rank~$1$ local system, $\C_{\rho}$, on our space $X$.  
Computing the homology groups of $X$ with coefficients 
in such local systems carves out some interesting subsets 
of the character group. 

\begin{definition}
\label{def:cvs}
The {\em characteristic varieties}\/ of $X$ are the sets 
\begin{equation}
\label{eq:cvs}
\VV^i_d(X)=\set{\rho \in H^1(X,\C^{\times})  
\mid \dim_{\C} H_i(X,\C_{\rho})\ge d}.
\end{equation}
\end{definition}

Clearly, $1\in \VV^i_d(X)$ if and only if $d\le b_i(X)$. 
In degree $0$, we have $\VV^0_1(X)= \{ 1\}$ 
and $\VV^0_d (X)=\emptyset$, for $d>1$. 
In degree $1$, the sets $\VV^i_d(X)$ depend only 
on the group $G=\pi_1(X,x_0)$---in fact, only on its 
maximal metabelian quotient, $G/G''$.

For the purpose of computing the characteristic varieties 
up to degree $i=k$, we may assume without loss of generality 
that $X$ is a finite CW-complex of dimension $k+1$, 
see \cite{PS-plms}.  With that in mind, it can be shown 
that the jump loci $\VV^i_d(X)$ are Zariski closed subsets 
of the algebraic group $H^1(X,\C^{\times})$, and that they 
depend only on the homotopy type of $X$.  
For details and further references, see \cite{Su11}.

One may extend the definition of characteristic varieties 
to arbitrary fields $\k$. The resulting varieties, $\VV^i_d(X,\k)$, 
behave well under field extensions: if $\k\subseteq \mathbb{K}$, 
then $\VV^i_d(X,\k)=\VV^i_d(X,\mathbb{K}) \cap H^1(X,\k^{\times})$. 

Most important for us are the depth one characteristic 
varieties, $\VV^i_1(X)$, and their unions up to a fixed 
degree, $\VV^i(X)=\bigcup_{j=0}^{i} \VV^j_1(X)$. 
These varieties yield an ascending filtration of 
the character group, 
\begin{equation}
\label{eq:filt asc}
\{1\}=\VV^0(X) \subseteq \VV^1(X) \subseteq  \cdots 
\subseteq \VV^k(X) \subseteq \wG.
\end{equation}

Now let $\wG^{\circ}$ be the identity 
component of the character group $\wG$. Writing 
$n=b_1(X)$, we may identify $\wG^{\circ}$ with the 
complex algebraic torus $(\C^{\times})^n$. 
Set 
\begin{equation}
\label{eq:wi}
\WW^i(X)=\VV^i(X)\cap \wG^{\circ}.  
\end{equation}
These varieties yield an ascending filtration of 
the complex algebraic torus $\wG^{\circ}$,
\begin{equation}
\label{eq:w filt asc}
\{1\}=\WW^0(X) \subseteq \WW^1(X) \subseteq  \cdots 
\subseteq \WW^k(X) \subseteq (\C^{\times})^n.
\end{equation}

The characteristic varieties behave well with respect to 
direct products.  For instance, suppose both $X_1$ and $X_2$ 
have finite $k$-skeleton.  Then, by \cite{PS-plms} 
(see also \cite{Su11}), we have that 
\begin{equation}
\label{eq:wprod}
\WW^i(X_1\times X_2)= 
\bigcup_{p+q=i} \WW^{p}(X_1) \times \WW^{q}(X_2),
\end{equation}
for all $i\le k$.

\subsection{Tangent cones and exponential tangent cones}
\label{subsec:tcone}

Let $W\subset (\C^{\times})^n$ be a Zariski closed subset, 
defined by an ideal $I$ in the Laurent polynomial ring   
$\C[t_1^{\pm 1},\dots , t_n^{\pm 1}]$.   
Picking a finite generating set for $I$, and multiplying 
these generators with suitable monomials if necessary, 
we see that $W$ may also be defined by the ideal $I\cap R$ 
in the polynomial ring $R=\C[t_1,\dots,t_n]$.  
Finally, let $J$ be the ideal  in the polynomial ring 
$S=\C[z_1,\dots, z_n]$, generated by the polynomials 
$g(z_1,\dots, z_n)=f(z_1+1, \dots , z_n+1)$, 
for all $f\in I\cap R$. 

\begin{definition}
\label{def:tcone}
The {\em tangent cone}\/ of $W$ at $1$ is the algebraic 
subset $\TC_1(W)\subset \C^n$ defined by the ideal 
$\init(J)\subset S$ generated by the initial forms of 
all non-zero elements from $J$.  
\end{definition}

The tangent cone $\TC_1(W)$ is a homogeneous 
subvariety of $\C^n$, which depends only on the analytic 
germ of $W$ at the identity.  In particular, 
$\TC_1(W)\ne \emptyset$ if and only if $1\in W$.  
Moreover, $\TC_1$ commutes with finite unions.

The following, related notion was introduced in \cite{DPS-duke}, 
and further studied in \cite{PS-plms} and \cite{Su11}.  

\begin{definition}
\label{def:exp tcone}
The {\em exponential tangent cone}\/ of $W$ at $1$ is 
the homogeneous subvariety $\tau_1(W)$ of $\C^n$, 
defined by 
\begin{equation}
\label{eq:tau1}
\tau_1(W)= \{ z\in \C^n \mid \exp(\lambda z)\in W,\ 
\text{for all $\lambda\in \C$} \}.
\end{equation}
\end{definition}

Again, $\tau_1(W)$ depends only on the analytic germ 
of $W$ at the identity,  and so $\tau_1(W)\ne \emptyset$ 
if and only if $1\in W$.  
Moreover, $\tau_1$ commutes with finite unions, as  
well as arbitrary intersections.  The most important 
properties of this construction are summarized in 
the following result from \cite{DPS-duke} (see also 
\cite{PS-plms} and \cite{Su11}).

\begin{theorem}
\label{thm:tau1}
For every Zariski closed subset $W\subset (\C^{\times})^n$, 
the following hold:
\begin{enumerate}
\item \label{t1}
$\tau_1(W)$ is a finite union 
of rationally defined linear subspaces of $\C^n$.
\item \label{t2}
$\tau_1(W)\subseteq \TC_1(W)$.
\end{enumerate}
\end{theorem}

If $W$ is an algebraic subtorus of $(\C^{\times})^n$, then 
$\tau_1(W)=\TC_1(W)$, and both types of tangent cones 
coincide with the tangent space at the origin, $T_1(W)$; 
moreover, $W=\exp( \tau_1(W) )$ in this case.   
More generally, if all positive-dimensional 
components of $W$ are algebraic subtori, then 
$\tau_1(W)=\TC_1(W)$.  In general, though, the 
inclusion from Theorem \ref{thm:tau1}\eqref{t2} can be strict.  

For brevity, we shall write $\tau_1^{\Q}(W)=\Q^n \cap \tau_1(W)$ 
for the rational points on the exponential tangent cone, and 
$\tau_1^{\R}(W)=\R^n \cap \tau_1(W)$ for the real points.

The main example we have in mind is that of the characteristic 
varieties $\WW^i(X)$, viewed as Zariski closed subsets of 
the algebraic torus $H^1(X,\C^{\times})^{\circ}=(\C^{\times})^n$, 
where $n=b_1(X)$.  
By Theorem  \ref{thm:tau1}, the exponential tangent cone 
to $\WW^i(X)$ can be written as a union of rationally defined 
linear subspaces, 
\begin{equation}
\label{eq:dfxi}
\tau_1( \WW^i(X)) =  
\bigcup_{L\in \CC_i(X)} L\otimes \C.   
\end{equation}
We call the resulting rational subspace arrangement, $\CC_i(X)$, 
the  {\em $i$-th characteristic arrangement}\/ of $X$; evidently, 
$\tau_1^{\Q}(\WW^i(X))$ is the union of this arrangement.

\subsection{Resonance varieties}
\label{subsec:rv}
Now consider the cohomology algebra $A=H^* (X,\C)$, with 
graded ranks the Betti numbers $b_i= \dim_{\C} A^i$. 
For each $a\in A^1$, we have $a^2=0$, by graded-commutativity 
of the cup product. Thus, right-multiplication by $a$ defines a 
cochain complex,
\begin{equation}
\label{eq:aomoto}
\xymatrix{(A , \cdot a)\colon  \ 
A^0\ar^(.66){a}[r] & A^1\ar^{a}[r] & A^2  \ar[r]& \cdots},
\end{equation}
The jump loci for the cohomology of this complex define a 
natural filtration of the affine space $A^1=H^1(X,\C)$. 

\begin{definition}
\label{def:rvs}
The {\em resonance varieties}\/ of $X$ are the sets 
\begin{equation}
\label{eq:rvs}
\RR^i_d(X)=\{a \in A^1 \mid \dim_{\C} 
H^i(A,a) \ge  d\}. 
\end{equation}
\end{definition}

For the purpose of computing the resonance varieties 
in degrees $i\le k$, we may assume without loss of 
generality that $X$ is a finite CW-complex of dimension $k+1$.  
The sets $\RR^i_d(X)$, then, are homogeneous, 
Zariski  closed subsets of $A^1=\C^n$, where $n=b_1$.  
In each degree $i\le k$, the resonance varieties provide 
a descending filtration,
\begin{equation}
\label{eq:res filt}
H^1(X,\C) \supseteq \RR^i_1(X) \supseteq \cdots 
\supseteq \RR^i_{b_i}(X) \supseteq \RR^i_{b_{i}+1}(X)=\emptyset.
\end{equation} 

Note that, if $A^i=0$, then $\RR^i_d(X)=\emptyset$, for all $d>0$. 
In degree $0$, we have $\RR^0_1(X)= \{ 0\}$,
and $\RR^0_d(X)= \emptyset$, for $d>1$. In degree $1$, 
the varieties $\RR^1_d(X)$ depend only on the fundamental 
group $G=\pi_1(X,x_0)$---in fact, only on the cup-product map 
$\cup \colon H^1(G,\C) \wedge H^1(G,\C) \to H^2(G,\C)$. 

One may extend the definition of resonance varieties 
to arbitrary fields $\k$, with the proviso that $H_1(X,\Z)$ 
should be torsion-free, if $\k$ has characteristic $2$. 
The resulting varieties, $\RR^i_d(X,\k)$, behave well 
under field extensions:  if $\k\subseteq \mathbb{K}$, then
$\RR^i_d(X,\k)=\RR^i_d(X,\mathbb{K}) \cap H^1(X,\k)$. 
In particular, $\RR^i_d(X,\Q)$ is just the set of rational 
points on the integrally defined variety 
$\RR^i_d(X)=\RR^i_d(X,\C)$. 

Most important for us are the depth-$1$ resonance varieties, 
$\RR^i_1(X)$, and their unions up to a fixed degree, 
$\RR^i(X)=\bigcup_{j=0}^{i} \RR^j_1(X)$. The latter 
varieties can be written as
\begin{equation}
\label{eq:depth1 res}
\RR^i(X)=\{a \in A^1 \mid 
 H^j(A,\cdot a) \ne 0, \text{ for some $j\le i$}\}. 
\end{equation}
These algebraic sets provide an ascending filtration of the 
first cohomology group, 
\begin{equation}
\label{eq:filt upres}
\{0\}=\RR^0(X) \subseteq \RR^1(X) \subseteq  \cdots \subseteq 
\RR^k(X) \subseteq H^1(X,\C)=\C^n.
\end{equation}

As noted in \cite{PS-plms}, the resonance varieties also 
behave well with respect to direct products: 
if both $X_1$ and $X_2$ have finite $k$-skeleton, then, 
for all $i\le k$,
\begin{equation}
\label{eq:rprod}
\RR^i(X_1\times X_2)= 
\bigcup_{p+q=i} \RR^{p}(X_1) \times \RR^{q}(X_2).
\end{equation}

\subsection{Tangent cone and resonance}
\label{subsec:tcone res}

An important feature of the theory of cohomology jumping 
loci is the relationship between characteristic and 
resonance varieties, based on the tangent cone 
construction. A foundational result in this direction 
is the following theorem of Libgober \cite{Li02}, 
which generalizes an earlier result of Green 
and Lazarsfeld \cite{GL87}. 

\begin{theorem}
\label{thm:lib}
Let $X$ be a connected CW-complex with finite $k$-skeleton. 
Then, for all $i\le k$ and $d>0$, the tangent cone at $1$ to 
$\WW_d^i(X)$ is included in $\RR_d^i(X)$. 
\end{theorem}
 
Putting together Theorems \ref{thm:tau1}\eqref{t2} and 
\ref{thm:lib}, and using the fact that both types of tangent 
cone constructions commute with finite unions, we 
obtain an immediate corollary. 

\begin{corollary}
\label{cor:tautcres}
For each $i\le k$, the following inclusions hold:
\begin{enumerate}
\item 
$\tau_1(\WW^i_d(X))\subseteq \TC_1(\WW^i_d(X))
\subseteq \RR^i_d(X)$, for all $d>0$.
\\[-8pt]
\item 
$\tau_1(\WW^i(X))\subseteq \TC_1(\WW^i(X))\subseteq \RR^i(X)$.
\end{enumerate}
\end{corollary}
In general, the above inclusions may very well be strict. 
In the presence of formality, though, they become equalities, 
at least in degree $i=1$. 

\subsection{Formality}
\label{subsec:formal}

Let us now collect some known facts on various formality 
notions. For further details and references, we refer to 
the recent survey \cite{PS-bmssmr}. 

Let $X$ be a connected CW-complex with finite 
$1$-skeleton.  The space $X$ is  {\em formal}\/ 
if there is a zig-zag of commutative, differential 
graded algebra quasi-isomorphisms connecting Sullivan's 
algebra of polynomial differential forms, $A_{\PL}(X, \Q)$,  
to the rational cohomology algebra, $H^*(X,\Q)$, 
endowed with the zero differential. The space $X$ is merely 
{\em $k$-formal}\/  (for some $k\ge 1$) if each of these 
morphisms induces an isomorphism in degrees up to $k$, 
and a monomorphism in degree $k+1$.

Examples of formal spaces include rational cohomology 
tori, Riemann surfaces, compact connected Lie groups, 
as well as their classifying spaces.  On the other hand, 
a nilmanifold is formal if and only if it is a torus. Formality 
is preserved under wedges and products of spaces, 
and connected sums of manifolds.  

The $1$-minimality property of a space  
depends only on its fundamental group.  
A finitely generated group $G$ is said to be $1$-formal 
if it admits a classifying space $K(G,1)$ which is $1$-formal, 
or, equivalently, if the Malcev Lie algebra $\m(G)$ (that is, the Lie 
algebra of the rational, prounipotent completion of $G$) admits a 
quadratic presentation.   Examples of $1$-formal groups include free 
groups and free abelian groups of finite rank, surface groups, 
and groups with first Betti number equal to $0$ or $1$.  
The $1$-formality property is preserved under free products 
and direct products.  

\begin{theorem}[\cite{DPS-duke}]
\label{thm:tcone}
Let $X$ be a $1$-formal space. Then, for each $d>0$,
\begin{equation}
\label{eq:tcone}
\tau_1(\VV^1_d(X))=\TC_1(\VV^1_d(X))=\RR^1_d(X).
\end{equation}
\end{theorem}

In particular, the first resonance variety, $\RR^1(X)$, of 
a $1$-formal space $X$ is a finite union of rationally 
defined linear subspaces.   

\subsection{Straightness}
\label{subsec:straight}
In \cite{Su-aspm}, we delineate another class of spaces 
for which the resonance and characteristic varieties are 
intimately related to each other via the tangent cone 
constructions. 

\begin{definition}
\label{def:straight}
We say that $X$ is {\em locally $k$-straight}\/ if, for each 
$i\le k$, all components of $\WW^{i}(X)$ passing through 
the origin $1$ are algebraic subtori, and the tangent cone 
at $1$ to $\WW^i(X)$ equals $\RR^{i}(X)$. If, moreover, 
all positive-dimensional components of $\WW^{i}(X)$ 
contain the origin, we say $X$ is {\em $k$-straight}. 
If these conditions hold for all $k\ge 1$, we say $X$ is  
{\em (locally) straight}.
\end{definition}

Examples of straight spaces include Riemann surfaces, 
tori, and knot complements. Under some further assumptions, 
the straightness properties behave well with respect 
to finite direct products and wedges.  

It follows from \cite{DPS-duke} that every $1$-formal 
space is locally $1$-straight.  In general, though, 
examples from \cite{Su-aspm} show that 
$1$-formal spaces need not be $1$-straight, and 
$1$-straight spaces need not be $1$-formal. 

\begin{theorem}[\cite{Su-aspm}]
\label{thm:rat res}
Let $X$ be a locally $k$-straight space.  Then, for all $i\le k$,
\begin{enumerate}
\item \label{rs1}
$\tau_1(\WW^i(X))=\TC_1(\WW^i(X))=\RR^i(X)$.
\\[-8pt]
\item \label{rs2}
$\RR^i(X,\Q)=
\bigcup_{L\in \CC_i(X)} L$. 
\end{enumerate}
\end{theorem}

In particular, all the resonance varieties $\RR^i(X)$ 
of a locally straight space $X$ are finite unions of 
rationally defined linear subspaces. 

\section{The Dwyer--Fried invariants}
\label{sec:df res}

In this section, we recall the definition of the 
Dwyer--Fried sets, and the way these sets relate 
to the (co)homology jump loci of a space.

\subsection{Betti numbers of free abelian covers}
\label{subsec:df}

As before, let $X$ be a connected CW-complex 
with finite $1$-skeleton, and let $G=\pi_1(X,x_0)$.  
Denote by $n=b_1(X)$ the first Betti number of 
$X$.  Fix an integer $r$ between $1$ and $n$, 
and consider the (connected) regular covers 
of $X$, with group of deck-transformations $\Z^r$.  

Each such cover, $X^{\nu}\to X$, is determined 
by an epimorphism $\nu \colon G \surj \Z^r$.  The 
induced homomorphism in rational cohomology, 
$\nu^*\colon H^1(\Z^r,\Q) \inj H^1(G,\Q)$, defines  
an $r$-dimensional subspace, $P_{\nu}=\im(\nu^*)$,  
in the $n$-dimensional $\Q$-vector space $H^1(G,\Q)=H^1(X,\Q)$.  
Conversely, each $r$-dimensional subspace 
$P\subset H^1(X,\Q)$ can be written as $P=P_{\nu}$, 
for some epimorphism $\nu \colon G \surj \Z^r$, 
and thus defines a regular $\Z^r$-cover of $X$. 

In summary, the regular $\Z^r$-covers of $X$ are 
parametrized by the Grassmannian of $r$-planes in 
$H^1(X,\Q)$, via the correspondence 
\begin{align}
\label{eq:corresp}
\big\{ \text{regular $\Z^r$-covers of $X$} \big\} 
& \longleftrightarrow 
\big\{ \text{$r$-planes in $H^1(X,\Q)$}\big\}\\
X^{\nu}  \to X\qquad & \longleftrightarrow\qquad  
P_{\nu}=\im(\nu^*).\notag
\end{align}

Moving about the rational Grassmannian 
and recording how the Betti numbers of the corresponding 
covers vary leads to the following definition.

\begin{definition}
\label{def:df}
The {\em Dwyer--Fried invariants}\/ of the space $X$ are 
the subsets 
\begin{equation*}
\label{eq:grass}
\Omega^i_r(X)=\big\{P_{\nu} \in \Grass_r(H^1(X,\Q)) \bigmid 
\text{$b_{j} (X^{\nu}) <\infty$ for $j\le i$}  \big\}. 
\end{equation*}
\end{definition}

For a fixed integer $r$ between $1$ and $n$, 
these sets form a descending filtration of the Grassmannian 
of $r$-planes in $H^1(X,\Q)=\Q^n$, 
\begin{equation}
\label{eq:df filt}
\Grass_r(\Q^n) = \Omega^0_r(X) \supseteq \Omega^1_r(X)  
\supseteq \Omega^2_r(X)  \supseteq \cdots.
\end{equation}
If $r>n$, we adopt the convention that $\Grass_r(\Q^n)=\emptyset$,  
and define $\Omega^i_r(X)=\emptyset$.

As noted in \cite{Su11}, the $\Omega$-sets are invariants 
of homotopy-type: if $f\colon X\to Y$ is a homotopy 
equivalence, then the induced isomorphism in cohomology, 
$f^*\colon H^1(Y,\Q) \to H^1(X,\Q)$, defines isomorphisms 
$f^*_r\colon \Grass_r(H^1(Y,\Q)) \to \Grass_r(H^1(X,\Q))$, 
which send each subset $\Omega^i_r(Y)$ bijectively 
onto $\Omega^i_r(X)$.  

Particularly simple is the situation when 
$n=b_1(X)>0$ and $r=n$.  In this case, $\Grass_n(H^1(X,\Q))=\pt$.  
Under the correspondence from \eqref{eq:corresp}, 
this single point is realized by the maximal free abelian cover, 
$X^{\alpha}\to X$, where  
$\alpha\colon  G\surj G_{\ab}/\Tors(G_{\ab})=\Z^n$ 
is the canonical projection. 
The sets $\Omega^i_n(X)$ are then given by
\begin{equation}
\label{eq:df b1}
\Omega^i_n(X)=\begin{cases}
\pt  & \text{if $b_j(X^{\alpha})<\infty$ for all $j\le i$},\\
\emptyset  & \text{otherwise}.\\
\end{cases}
\end{equation}

\subsection{Dwyer--Fried invariants and characteristic varieties}
\label{subsec:df cv}

The next theorem reduces the computation of the $\Omega$-sets 
to a more standard computation in algebraic geometry.  
The theorem was proved by Dwyer and Fried in \cite{DF}, 
using the support loci for the Alexander invariants, 
and was recast in a slightly more general context by 
Papadima and Suciu in \cite{PS-plms}, using the 
characteristic varieties.  We state this result in the  
form established in \cite{Su11}.

\begin{theorem}[\cite{DF, PS-plms, Su11}]
\label{thm:df cv}
Suppose $X$ has finite $k$-skeleton, for some $k\ge 1$.  
Then, for all $i\le k$ and $r\ge 1$, 
\begin{equation}
\label{eq:grass cv}
\Omega^i_r(X)=\big\{P\in \Grass_r(\Q^n) \bigmid
\dim \, (\exp(P \otimes \C) \cap \WW^i(X) )=0 \big\}.
\end{equation}
\end{theorem}

In other words, an $r$-plane $P\subset \Q^n$ belongs 
to the set $\Omega^i_r(X)$ if and only if the algebraic torus 
$T=\exp(P \otimes \C)$ intersects the characteristic 
variety $W=\WW^i(X)$ only in finitely many points.  
When this happens, the exponential tangent cone 
$\tau_1(T\cap W)$ equals $\{0\}$, forcing $P\cap L=\{0\}$, 
for every subspace $L\subset \Q^n$ in the characteristic 
subspace arrangement $\CC_i(X)$.  Consequently,  
\begin{equation}
\label{eq:ubound}
\Omega^i_r(X) \subseteq \bigg(
\bigcup_{L\in \CC_i(X)} \big\{P \in \Grass_r(H^1(X,\Q)) \bigmid 
P\cap L \ne \{0\} \big\} \bigg)^{\compl}.
\end{equation}

As noted in \cite{Su11}, this inclusion may be 
reinterpreted in terms of the classical 
incidence correspondence from algebraic geometry. 
Let $V$ be a homogeneous variety in $\k^n$.  
The locus of $r$-planes in $\k^n$ that meet $V$, 
\begin{equation}
\label{eq:incident}
\sigma_r(V) = \big\{ P \in \Grass_r(\k^n) 
\bigmid P \cap  V \ne \{0\} \big\}, 
\end{equation}
is a Zariski closed subset of the Grassmannian 
$\Grass_r(\k^n)$.  In the case when $V$ is a 
linear subspace $L\subset \k^n$, the incidence 
variety $\sigma_r(L)$ is known as the  
{\em special Schubert variety}\/ defined by $L$. 
If $L$ has codimension $d$ in $\k^n$, then 
$\sigma_r(L)$ has codimension $d-r+1$ in $\Grass_r(\k^n)$. 

Applying this discussion to the homogeneous variety 
$\tau^{\Q}_1 (\WW^i(X)) = \bigcup_{L\in \CC_i(X)} L$ 
lying inside $H^1(X,\Q)=\Q^n$, and using 
formula \eqref{eq:ubound}, we obtain the following 
corollary.

\begin{corollary}[\cite{Su11}]
\label{cor:omega schubert}
Let $X$ be a connected CW-complex with finite $k$-skeleton. 
For all $i\le k$ and $r\ge 1$, 
\begin{align}
\label{eq:omega schubert}
\Omega^i_r(X) &\subseteq  \Grass_r(H^1(X,\Q)) \setminus 
 \sigma_r \big( \tau^{\Q}_1 (\WW^i(X)) \big)\\
 \notag
 &=   \Grass_r(H^1(X,\Q)) \setminus 
\bigcup_{L\in \CC_i(X)} \sigma_r(L ).
\end{align}
\end{corollary}

In other words, each set $\Omega^i_r(X)$ is 
contained in the complement of a Zariski closed subset 
of the Grassmanian $\Grass_r(H^1(X,\Q))$, namely, the 
union of the special Schubert varieties $\sigma_r(L)$ 
corresponding to the subspaces $L$ in $\CC_i(X)$. 

Under appropriate hypothesis, the inclusion from 
Corollary \ref{cor:omega schubert} holds as equality. 
The next two propositions illustrate this point. 

\begin{prop}[\cite{Su11}]
\label{prop:tau schubert}
Let $X$ be a connected CW-complex with finite $k$-skeleton. 
Suppose that, for some $i\le k$, all positive-dimensional components 
of $\WW^i(X)$ are algebraic subtori.  Then, for all $r\ge 1$,  
\begin{equation}
\label{eq:sharp schubert}
\Omega^i_r(X) = \Grass_r(H^1(X,\Q)) \setminus 
\bigcup_{L\in \CC_i(X)} \sigma_r(L ).
\end{equation}
\end{prop}

\begin{prop}[\cite{DF},  \cite{PS-plms}, \cite{Su11}] 
\label{prop:df1}
Let $X$ be a CW-complex with finite $k$-skeleton. 
Then, for all $i\le k$, 
\begin{equation}
\label{eq:df1}
\Omega^i_1(X) = \bP(H^1(X,\Q))\setminus 
\bigcup_{L\in \CC_i(X)} \bP(L),
\end{equation}
where $\bP(V)$ denotes the projectivization 
of a homogeneous subvariety $V\subseteq H^1(X,\Q)$.
\end{prop}

In either of these two situations, the sets $\Omega^i_r(X)$ 
are Zariski open subsets of $\Grass_r(H^1(X,\Q))$. In 
general, though, the sets $\Omega^i_r(X)$ need not 
be open, not even in the usual topology on the rational 
Grassmanian. This phenomenon was first noticed by 
Dwyer and Fried, who constructed in \cite{DF} a 
$3$-dimensional cell complex $X$ for which 
$\Omega^2_2(X)$ is a finite set (see Example \ref{ex:df}). 
In \cite{Su11}, we  provide examples of finitely presented 
groups $G$ for which $\Omega^1_2(G)$ is not open.  

\subsection{Straightness and the $\Omega$-sets}
\label{subsec:df straight}

Under appropriate straightness assumptions, the 
upper bounds for the Dwyer--Fried sets can be 
expressed in terms of the resonance varieties associated 
to the cohomology ring $H^*(X,\Q)$.

\begin{theorem}[\cite{Su-aspm}]
\label{thm:df straight}
Let $X$ be a connected CW-complex. 

\begin{enumerate}
\item \label{s1}
If $X$ is locally $k$-straight, then
$\Omega^i_r(X) \subseteq \Grass_r(H^1(X,\Q)) \setminus 
\sigma_r(\RR^i(X,\Q))$,  
for all $i\le k$ and $r\ge 1$. 
\item \label{s2}
If $X$ is $k$-straight, then 
$\Omega^i_r(X) = \Grass_r(H^1(X,\Q)) \setminus 
\sigma_r(\RR^i(X,\Q))$, 
for all $i\le k$ and $r\ge 1$. 
\end{enumerate}
\end{theorem}

If $X$ is locally $k$-straight, we also know from 
Theorem \ref{thm:rat res}\eqref{rs2} that $\RR^i(X,\Q)$ 
is the union of the linear subspaces comprising $\CC_i(X)$, 
for all $i\le k$.   Thus, if $X$ is $k$-straight, then $\Omega^i_r(X)$ 
is the complement of a finite union of special Schubert 
varieties in the Grassmannian of $r$-planes in $H^1(X,\Q)$;   
in particular, $\Omega^i_r(X)$ is a Zariski open set in 
$\Grass_r(H^1(X,\Q))$.

The straightness hypothesis is crucial for the equality in 
Theorem \ref{thm:df straight}\eqref{s2} to hold. 

\begin{example}
\label{ex:straight c again}
Let $G=\langle x_1, x_2 \mid x_1^2 x_2 = x_2 x_1^2 \rangle$. 
Then $\WW^1(G)=\{1\}\cup \{t\in (\C^{\times})^2 \mid t_1=-1\}$, 
while $\RR^1(G)=\{0\}$.  Thus, $G$ is locally $1$-straight, 
but not $1$-straight.  Moreover, $\Omega^1_2(G)=\emptyset$, 
yet $\sigma_2(\RR^1(G,\Q))^{\compl}=\pt$. 
\end{example}

\section{The Bieri--Neumann--Strebel--Renz invariants}
\label{sec:bnsr}

We now go over the several definitions of the $\Sigma$-invariants 
of a space $X$ (and, in particular, of a group $G$), and discuss 
the way these invariants relate to the (co)homology jumping loci. 
 
\subsection{A finite type property}
\label{subsec:fgs}
We start with a finiteness condition for chain complexes, 
following the approach of Farber, Geoghegan, and 
Sch\"utz \cite{FGS}.   Let $C=(C_i,\partial_i)_{i\ge 0}$ 
be a non-negatively graded chain complex over a ring $R$, 
and let $k$ be a non-negative integer.   

\begin{definition}
\label{def:ktype}
We say $C$ is of {\em finite $k$-type}\/ 
if there is a  chain complex $C'$ of finitely generated, projective 
$R$-modules and a $k$-equivalence between $C'$ and $C$, 
i.e., a chain map $C'\to C$ inducing isomorphisms 
$H_i(C')\to H_i(C)$ for $i<k$ and an epimorphism 
$H_k(C')\to H_k(C)$. 
\end{definition}

Equivalently, there is a chain complex of 
free $R$-modules, call it $D$, such that $D_i$ is finitely 
generated for all $i\le k$, and there is a chain map $D\to C$ 
inducing an isomorphism $H_*(D)\to H_*(C)$.  When $C$ 
itself is free, we have the following alternate characterization 
from \cite{FGS}.

\begin{lemma}
\label{lem:ktype}
Let $C$ be a free chain complex over a ring $R$.  
Then $C$ is of finite $k$-type if and only if $C$ 
is chain-homotopy equivalent to a chain complex $D$ 
of free $R$-modules, such that $D_i$ 
is finitely generated for all $i\leq k$.
\end{lemma}

\begin{remark}
\label{rem:ktype}
Suppose $\rho\colon R\to S$ is a ring morphism, and 
$C$ is a chain complex over $R$.  Then $C\otimes_{R} S$ 
naturally acquires the structure of an $S$-chain complex via 
extension of scalars.  Now, if $C$ is free, and of finite $k$-type 
over $R$, it is readily seen that $C\otimes_{R} S$ is of finite 
$k$-type over $S$. 
\end{remark}

\subsection{The $\Sigma$-invariants of a chain complex}
\label{subsec:bnsr cc}
Let $G$ be a finitely generated group, and denote by  
$\Hom(G,\R)$ the set of homomorphisms from $G$ to the additive 
group of the reals.  Clearly, this is a finite-dimensional 
$\R$-vector space, on which the multiplicative group of positive 
reals naturally acts.  After fixing an inner product on $\Hom(G,\R)$, 
the quotient space, 
\begin{equation}
\label{eq:sg}
S(G)= (\Hom(G,\R)\setminus\{0\})/\R^{+},
\end{equation}
may be identified with the unit sphere in $\Hom(G,\R)$. 
Up to homeomorphism, this sphere is determined 
by the first Betti number of $G$.  Indeed,  if $b_1(G)=n$, 
then $S(G)=S^{n-1}$; in particular, if $b_1(G)=0$, then 
$S(G)=\emptyset$.   To simplify notation, we will denote 
both a non-zero homomorphism $G\to \R$ and its 
equivalence class in $S(G)$ by the same symbol, 
and we will routinely view $S(G)$ as embedded in 
$\Hom(G,\R)$. 

Given a homomorphism $\chi\colon G\to \R$, consider the set  
$G_{\chi}=\{ g \in G \mid \chi(g)\ge 0\}$.  Clearly, $G_{\chi}$ 
is a submonoid of $G$, and the monoid ring 
$\Z{G}_{\chi}$ is a subring of the group ring $\Z{G}$. 
Thus, any $\Z{G}$-module naturally acquires the structure of a 
$\Z{G}_{\chi}$-module, by restriction of scalars.

\begin{definition}[\cite{FGS}]  
\label{def:sigma chain}
Let $C$ be a chain complex over $\Z{G}$.  For each 
integer $k\ge 0$, the {\em $k$-th 
Bieri--Neumann--Bieri--Strebel invariant}\/ of $C$ 
is the set
\begin{equation}
\label{eq:sigmakc}
\Sigma^k(C)=\{\chi\in S(G) \mid 
\text{$C$ is of finite $k$-type over $\Z{G}_{\chi}$}\}.
\end{equation}
\end{definition}
Note that $\Z{G}$ is a flat $\Z{G}_\chi$-module, and  
$\Z{G}\otimes_{\Z{G}_\chi} M \cong M$ for every 
$\Z{G}$-module $M$.  Thus, if $\Sigma^k(C)$ is 
non-empty, then $C$ must be of finite $k$-type 
over $\Z{G}$.

To a large extent, the importance of the $\Sigma$-invariants 
lies in the fact that they control the finiteness properties 
of kernels of projections to abelian quotients.  More precisely, 
let $N$ be a normal subgroup of $G$, with $G/N$ abelian. 
Define  
\begin{equation}
\label{eq:sgn}
S(G,N)=\{\chi \in S(G) \mid N\le \ker (\chi) \}.  
\end{equation}

It is readily seen that  $S(G,N)$ is the great subsphere 
obtained by intersecting the unit sphere $S(G)\subset H^1(G,\R)$ 
with the linear subspace $P_{\nu}\otimes \R$, where  
$\nu\colon G\surj G/N$ is the canonical projection, and 
$P_{\nu}=\im(\nu^*\colon H^1(G/N,\Q) \inj H^1(G,\Q))$.  
Notice also that every $\Z{G}$-module  
acquires the structure of a $\Z{N}$-module by restricting 
scalars.

\begin{theorem}[\cite{FGS}] 
\label{thm:fgs}
Let $C$ be a chain complex of free $\Z{G}$-modules, with 
$C_i$ finitely generated for $i\le k$, and let $N$ be a 
normal subgroup of $G$, with $G/N$ is abelian. 
Then $C$ is of finite $k$-type over $\Z{N}$ 
if and only if $S(G,N)\subset \Sigma^k(C)$.
\end{theorem}

\subsection{The $\Sigma$-invariants of a CW-complex}
\label{subsec:fgs cw}
Let $X$ be a connected CW-complex with 
finite $1$-skeleton, and let $G=\pi_1(X,x_0)$ be its fundamental 
group.  Picking a classifying map $X\to K(G,1)$, we obtain 
an induced isomorphism, $H^1(G,\R) \isom H^1(X,\R)$, 
which identifies the respective unit spheres, 
$S(G)$ and $S(X)$.

The cell structure on $X$ lifts to a cell structure on the 
universal cover $\wX$.  Clearly, this lifted cell structure 
is invariant under the action of $G$ by deck transformations. 
Thus, the cellular chain complex $C_*(\wX,\Z)$ 
is a chain complex of (free) $\Z{G}$-modules.  

\begin{definition}
\label{def:fgs cw}
For each $k\ge 0$, the {\em $k$-th 
Bieri--Neumann--Strebel--Renz invariant}\/ 
of $X$ is the subset of $S(X)$ given by
\begin{equation}
\label{eq:sigmakx}
\Sigma^k(X,\Z)=\Sigma^k(C_*(\wX,\Z)).
\end{equation}
\end{definition}

It is shown in \cite{FGS} that  $\Sigma^k(X,\Z)$ is an open 
subset of $S(X)$, which depends only on the homotopy 
type of the space $X$.  Clearly, $\Sigma^0(X,\Z)=S(X)$.

Now let $G$ be a finitely generated group, and pick a 
classifying space $K(G,1)$.   In view of the above 
discussion, the sets 
\begin{equation}
\label{eq:sigma G}
\Sigma^k(G,\Z):=\Sigma^k(K(G,1),\Z), 
\end{equation} 
are well-defined invariants of the group $G$.  These sets, which 
live inside the unit sphere $S(G)\subset \Hom(G,\R)$,  
coincide with the classical geometric invariants of 
Bieri, Neumann and Strebel \cite{BNS} and Bieri 
and Renz \cite{BR}.

The $\Sigma$-invariants of a CW-complex and those 
of its fundamental group are related, as follows.

\begin{prop}[\cite{FGS}]
\label{prop:sigma xg}
Let $X$ be a connected CW-complex with finite $1$-skeleton.  
If $\wX$ is $k$-connected, then 
\begin{equation}
\label{eq:sigma xg}
\Sigma^k(X)=\Sigma^k(\pi_1(X)), \text{ and } \ 
\Sigma^{k+1}(X)\subseteq \Sigma^{k+1}(\pi_1(X)).
\end{equation}
\end{prop}
The last inclusion can of course be strict.  For 
instance, if $X=S^1\vee S^{k+1}$, with $k\ge 1$, 
then $\Sigma^{k+1}(X)=\emptyset$, 
though $\Sigma^{k+1}(\pi_1(X))=S^0$.  We 
shall see a more subtle occurrence of this 
phenomenon in Example \ref{ex:df}.

\subsection{Generalizations and discussion}
\label{subsec:sigmag}
More generally, if $M$ is a $\Z{G}$-module, the invariants 
$\Sigma^{k}(G, M)$ of Bieri and Renz \cite{BR} are given by 
$\Sigma^{k}(G, M)=\Sigma^i(F_{\bullet})$, where 
$F_{\bullet} \to M$ is a projective $\Z{G}$-resolution of $M$.  
In particular, we have the invariants $\Sigma^{k}(G,\k)$, 
where $\k$ is a field, viewed as a trivial $\Z{G}$-module.  
There is always an inclusion 
$\Sigma^{k}(G,\Z) \subseteq \Sigma^i(G,\k)$, 
but this inclusion may be strict, as we shall see in 
Example \ref{ex:flag rp2} below.   

An alternate definition of the $\Sigma$-invariants of 
a group is as follows. 
Recall that a monoid (in particular, a group) $G$ 
is of type $\FP_k$ if there is a projective $\Z{G}$-resolution 
$F_{\bullet}\to \Z$, with $F_i$ finitely generated, for all $i\le k$. 
In particular, $G$ is of type $\FP_1$ if and only if $G$ is finitely 
generated. We then have 
\begin{equation}
\label{eq:bnsr def}
\Sigma^k(G,\Z)=\set{\chi\in S(G) \mid 
\text{the monoid $G_{\chi}$ is of type $\FP_k$}}.
\end{equation}
These sets form a descending chain of open 
subsets of $S(G)$, starting at $\Sigma^1(G)=\Sigma^1(G,\Z)$.  
Moreover, $\Sigma^k(G,\Z)$ is non-empty only if 
$G$ is of type $\FP_k$.   

If $N$ is a normal subgroup of $G$, with $G/N$ abelian, 
then Theorem \ref{thm:fgs} implies the following result 
from \cite{BNS, BR}:  The group $N$ is of type $\FP_k$ 
if and only if  $S(G,N)\subseteq \Sigma^k(G, \Z)$. 
In particular, the kernel of an epimorphism $\chi\colon G \surj \Z$ is 
finitely generated if and only if both $\chi$ and $-\chi$ belong 
to $\Sigma^1(G)$.

One class of groups for which the BNSR invariants can be 
computed explicitly is that of one-relator groups. 
K.~Brown gave in \cite{Br87} an algorithm for 
computing $\Sigma^1(G)$ for groups $G$ in this class,  
while Bieri and Renz in \cite{BR} reinterpreted this algorithm 
in terms of Fox calculus, and showed that, for $1$-relator 
groups, $\Sigma^{k}(G,\Z)=\Sigma^1(G)$, for all $k\ge 2$. 

Another class of groups for which the $\Sigma$-invariants 
can be completely determined is that of non-trivial free products.  
Indeed, if $G_1$ and $G_2$ are two non-trivial, 
finitely generated groups, then $\Sigma^{k}(G_1*G_2,\Z)=\emptyset$, 
for all $k\ge 1$ (see, for instance, \cite{PS-plms}). 

Finally, it should be noted that the BNSR invariants 
obey some very nice product formulas.  
For instance, let $G_1$ and $G_2$ be two groups of type 
$\FF_k$, for some $k\ge 1$.  Identify the sphere 
$S(G_1\times G_2)$ with the join $S(G_1)*S(G_2)$. 
Then, for all $i\le k$, 
\begin{equation}
\label{eq:sigmaprod1} 
\Sigma^i(G_1\times G_2,\Z)^{\compl} \subseteq
\bigcup_{p+q=i} \Sigma^{p}(G_1,\Z)^{\compl} *  
\Sigma^{q}(G_2,\Z)^{\compl}, 
\end{equation}
where again $A*B$ denotes the join of two spaces, 
with the convention that $A*\emptyset =A$. 
As shown  in 
\cite[Theorem 7.4]{BNS}, the above inclusion 
holds as equality for $i=1$, i.e.,
\begin{equation}
\label{eq:sigmaprod2} 
\Sigma^1(G_1\times G_2,\Z)^{\compl} =
\Sigma^{1}(G_1,\Z)^{\compl}  \cup   
\Sigma^{1}(G_2,\Z)^{\compl}. 
\end{equation}

The general formula \eqref{eq:sigmaprod1} was established by 
Meinert (unpublished) and Gehrke \cite{Ge}.  Recently, it was 
shown by Sch\"{u}tz \cite{Sc} and Bieri--Geoghegan \cite{BiG} 
that equality holds in \eqref{eq:sigmaprod1} for all 
$i\le 3$, although equality may fail for $i\ge 4$. 
Furthermore, it was shown in \cite{BiG} that 
the analogous product formula for the $\Sigma$-invariants 
with coefficients in a field $\k$ holds as an equality, for all $i\le k$.

\subsection{Novikov homology}
\label{subsec:novikov}
In his 1987 thesis, J.-Cl.~Sikorav reinterpreted the BNS invariant 
of a finitely generated group in terms of Novikov homology. 
This interpretation was extended to all BNSR invariants 
by Bieri \cite{Bi07}, and later to the BNSR invariants of 
CW-complexes by Farber, Geoghegan and Sch\"{u}tz \cite{FGS}. 

The {\em Novikov--Sikorav completion}\/ of the group ring $\Z{G}$ 
with respect to a homomorphism $\chi\colon G\to \R$ consists of 
all formal sums $\sum_j n_j g_j$, with $n_j\in \Z$ and 
$g_j\in G$, having the property that, for each $c\in \R$, the set 
of indices $j$ for which $n_j \ne 0 $ and $\chi(g_j) \ge c$ is finite. 
With the obvious addition and multiplication, the Novikov--Sikorav 
completion, $\widehat{\Z{G}}_{\chi}$, is a ring, containing $\Z{G}$ 
as a subring; in particular,  $\widehat{\Z{G}}_{\chi}$ carries a natural 
$G$-module structure.  We refer to M.~Farber's book \cite{Far} for 
a comprehensive treatment, and to R.~Bieri \cite{Bi07} for further details. 

Let $X$ be a connected CW-complex with finite $k$-skeleton, 
and let $G=\pi_1(X,x_0)$ be its fundamental group.  

\begin{theorem}[\cite{FGS}]
\label{thm:bns novikov}
With notation as above, 
\begin{equation}
\label{eq:bnsr fgs}
\Sigma^{k}(X, \Z)= \{ \chi \in S(X) \mid
H_{i}(X, \widehat{\Z{G}}_{-\chi})=0,\: \text{for all $i\le k$} \}.
\end{equation}
\end{theorem}

Now, every non-zero homomorphism $\chi\colon G\to \R$ factors 
as $\chi=\iota \circ \xi$, where $\xi\colon G\surj \Gamma$ is a 
surjection onto a lattice $\Gamma\cong \Z^r$ in $\R$, 
and $\iota\colon \Gamma\inj \R$ is the inclusion map.  
A Laurent polynomial  $p=\sum_{\gamma} n_{\gamma} \gamma\in 
\Z\G$ is said to be $\iota$-monic if the greatest 
element in $\iota(\supp(p))$ is $0$, and $n_0=1$;     
every such polynomial is invertible in the completion 
$\widehat{\Z\Gamma}_{\iota}$. 
We denote by $\rat$ the localization of $\Z\G$ 
at the multiplicative subset of all $\iota$-monic 
polynomials.  Using the known fact that $\rat$ is both 
a $G$-module and a PID, one may define for each $i\le k$ the 
{\em $i$-th Novikov Betti number}, $b_i(X,\chi)$,  
as the rank of the finitely generated $\rat$-module $H_i(X,\rat)$.  

\subsection{$\Sigma$-invariants and characteristic varieties}
\label{subsec:bns bound}
 
The following result from \cite{PS-plms} creates a bridge  
between the $\Sigma$-invariants of a space and 
the real points on the exponential tangent cones 
to the respective characteristic varieties.

\begin{theorem}[\cite{PS-plms}]
\label{thm:nov betti}
Let $X$ be a CW-complex with finite $k$-skeleton, for some 
$k\ge 1$, and let $\chi\in S(X)$. The following then hold.
\begin{enumerate}
\item 
If $-\chi$ belongs to $\Sigma^k(X,\Z)$, then $H_i(X,\rat)=0$, 
and so $b_i(X,\chi)=0$, for all $i\le k$. 
\item
$\chi$ does not belong to $\tau^{\R}_1(\WW^k(X))$ if and only if 
$b_i(X,\chi)=0$, for all $i\le k$.
\end{enumerate}
\end{theorem}

Recall that $S(V)$ denotes the intersection of the unit sphere 
$S(X)$ with a homogeneous subvariety $V\subset H^1(X,\R)$.  
In particular, if $V=\{0\}$, then $S(V)=\emptyset$. 

\begin{corollary}[\cite{PS-plms}]
\label{cor:bns tau}
Let $X$ be a CW-complex with finite $k$-skeleton. Then, 
for all $i\le k$,
\begin{equation}
\label{eq:bns bound1}
\Sigma^i(X, \Z)\subseteq S(X) \setminus 
S( \tau^{\R}_1(\WW^i(X)) ). 
\end{equation}
\end{corollary}

Qualitatively, Corollary \ref{cor:bns tau} says that each BNSR set
$\Sigma^i(X, \Z)$ is contained in the complement of a union of 
rationally defined great subspheres.   

As noted in \cite{PS-plms}, the above bound is sharp. 
For example, if $X$ is a nilmanifold, then $\Sigma^i(X, \Z) = S(X)$, 
while $\WW^i(X,\C)=\{1\}$, and so $\tau^{\R}_1(\WW^i(X))=\{0\}$, 
for all $i\ge 1$.  Thus, the inclusion from Corollary \ref{cor:bns tau}
holds as an equality in this case. 

If the space $X$ is (locally) straight, Theorem \ref{thm:rat res} 
allows us to replace the exponential tangent cone 
by the corresponding resonance variety.  

\begin{corollary}
\label{cor:bns straight}
If $X$ is locally $k$-straight, then, for all $i\le k$,
\begin{equation}
\label{eq:sigmastraight}
\Sigma^i(X, \Z)\subseteq S(X) \setminus S( \RR^i(X,\R)).
\end{equation}
\end{corollary}

\begin{corollary}[\cite{PS-plms}]
\label{cor:bns 1f}
If $G$ is a $1$-formal group, then 
\begin{equation}
\label{eq:bns res bound}
\Sigma^1(G) \subseteq S(G) \setminus S( \RR^1(G,\R)).
\end{equation}
\end{corollary}

\section{Relating the $\Omega$-invariants and the $\Sigma$-invariants}
\label{sect:omega sigma}

In this section, we prove Theorem \ref{thm:intro main} from 
the Introduction, which essentially says the following:  
if inclusion \eqref{eq:omega schubert} is strict, then 
inclusion \eqref{eq:bns bound1} is also strict. 

\subsection{Finiteness properties of abelian covers}
\label{subsec:finite abel}

Let $X$ be a connected CW-complex with finite 
$1$-skeleton, and let $\wX$ be the universal cover, with 
group of deck transformation $G=\pi_1(X,x_0)$. 

Let $p\colon X^{\nu}\to X$ be a (connected) regular, 
free abelian cover, associated to an epimorphism 
$\nu\colon  G\surj \Z^r$.  Fix a basepoint $\tilde{x}_0\in p^{-1}(x_0)$, 
and identify the fundamental group $\pi_1(X^{\nu}, \tilde{x}_0)$ 
with $N=\ker(\nu)$.  Note that the universal cover $\widetilde{X^{\nu}}$ 
is homeomorphic to $\wX$. 

Finally, set
\begin{equation}
\label{eq:s cover}
S(X,X^{\nu})=\{\chi\in S(X) \mid p^*(\chi)=0\}.
\end{equation}
Then $S(X,X^{\nu})$ is a great sphere of 
dimension $r-1$.   In fact, 
\begin{equation}
\label{eq:sxxp}
S(X,X^{\nu})=S(X)\cap (P_{\nu}\otimes \R), 
\end{equation}
where recall $P_{\nu}$ is the $r$-plane in $H^1(X,\Q)$ 
determined by $\nu$ via the correspondence from \eqref{eq:corresp}. 
Moreover, under the identification $S(X)=S(G)$, the subsphere 
$S(X,X^{\nu})$ corresponds to $S(G,N)$.

\begin{prop}
\label{prop:fgs cw}
Let $X$ be a connected CW-complex with finite $k$-skeleton, 
and let $p\colon X^{\nu}\to X$ be a regular, free abelian cover. 
Then the chain complex $C_*(\widetilde{X^{\nu}},\Z)$ is of finite 
$k$-type over $\Z{N}$ if and only if 
$S(X,X^{\nu})\subseteq \Sigma^k(X,\Z)$.
\end{prop}

\begin{proof}
Consider the free $\Z{G}$-chain complex $C=C_*(\widetilde{X},\Z)$.  
Upon restricting scalars to the subring $\Z{N}\subset \Z{G}$, the 
resulting $\Z{N}$-chain complex may be identified with 
$C_*(\widetilde{X^{\nu}},\Z)$.  The desired conclusion follows 
from Theorem \ref{thm:fgs}.
\end{proof}

\subsection{Upper bounds for the $\Sigma$- and $\Omega$-invariants}
\label{subsec:bounds}
Now assume $X$ has finite $k$-skeleton, for some $k\ge 1$.
As we just saw, great subspheres in the $\Sigma$-invariants 
indicate directions in which the corresponding covers have good 
finiteness properties. The next result compares the rational 
points on such subspheres with the corresponding 
$\Omega$-invariants. 

\begin{prop} 
\label{prop:bns cv}
Let $P\subseteq  H^1(X,\Q)$ be an $r$-dimensional 
linear subspace.  If the unit sphere in $P$ is 
included in $\Sigma^k(X, \Z)$, then 
$P$ belongs to $\Omega^k_r(X)$.
\end{prop}

\begin{proof}
Realize $P=P_{\nu}$, for some epimorphism $\nu\colon G\surj \Z^r$, 
and let $X^{\nu} \to X$ be the corresponding $\Z^r$-cover.  
By assumption, $S(X,X^{\nu}) \subseteq \Sigma^k(X,\Z)$.  
Thus, by Proposition \ref{prop:fgs cw}, the chain complex 
$C_*(\widetilde{X^{\nu}},\Z)$ is of finite $k$-type over $\Z{N}$.  

Now, in view of Remark \ref{rem:ktype}, the chain complex 
$C_*(X^{\nu},\Z)=C_*(\widetilde{X^{\nu}},\Z)\otimes_{\Z{N}} \Z$ 
is of finite $k$-type over $\Z$.  Therefore, $b_i(X^{\nu})<\infty$, 
for all $i\le k$.  Hence, $P\in \Omega^k_r(X)$. 
\end{proof}

\begin{corollary} 
\label{cor:ptau}
Let $P\subseteq  H^1(X,\Q)$ be a  linear subspace. 
If  $S(P) \subseteq \Sigma^k(X, \Z)$, 
then  $P \cap \tau_1^{\Q}(\WW^k(X))  =\{0\}$.
\end{corollary}

\begin{proof}
Set $r=\dim_{\Q} P$. 
From Corollary \ref{cor:omega schubert}, we know that 
$\Omega^k_r(X)$ is contained in the complement 
in $\Grass_r(H^1(X,\Q))$ to the incidence variety 
$ \sigma_r ( \tau^{\Q}_1 (\WW^k(X)) )$.  
The conclusion follows from Proposition \ref{prop:bns cv}. 
\end{proof}

We are now in a position to state and prove the 
main result of this section, which is simply a 
restatement of Theorem \ref{thm:intro main} 
from the Introduction.

\begin{theorem} 
\label{thm:bns omega}
If $X$ is a connected CW-complex with finite $k$-skeleton, then, 
for all $r\ge 1$,
\begin{equation}
\label{eq:bns omega}
\Sigma^k(X,\Z) = S(\tau^{\R}_1 (\WW^k(X)))^{\compl} 
\implies 
\Omega^k_r(X) = \sigma_r ( \tau^{\Q}_1 (\WW^k(X)) )^{\compl} . 
\end{equation}
\end{theorem}

\begin{proof}
From Corollary \ref{cor:omega schubert}, we know that 
$\Omega^k_r(X) \subseteq\sigma_r ( \tau^{\Q}_1 (\WW^k(X)) )^{\compl}$. 
Suppose this inclusion is strict.  There is then an $r$-dimensional 
linear subspace $P\subseteq H^1(X,\Q)$ such that 
\begin{enumerate} 
\item \label{y1}
$P\cap  \tau^{\Q}_1 (\WW^k(X))=\{0\}$, and 
\item \label{y2} 
$P\notin \Omega^k_r(X)$. 
\end{enumerate}

By supposition \eqref{y2}  and Proposition \ref{prop:bns cv}, 
we must have $S(P)\not\subseteq \Sigma^k(X,\Z)$.  Thus, 
there exists an element $\chi \in S(P)$ such that 
$\chi\notin \Sigma^k(X,\Z)$. 

Now, supposition \eqref{y1} and the fact that $\chi \in P$ 
imply that  $\chi \notin \tau^{\Q}_1 (\WW^k(X))$.  Since 
$\chi$ belongs to $H^1(X,\Q)$, we infer that 
$\chi\notin \tau^{\R}_1 (\WW^k(X))$.

We have shown that $\Sigma^k(X,\Z) \subsetneqq 
S(\tau^{\R}_1 (\WW^k(X)))^{\compl}$, 
a contradiction.
\end{proof}

Recall now that the incidence variety 
$\sigma_r ( \tau^{\Q}_1 (\WW^k(X)) )$ is a Zariski closed 
subset of the Grassmannian $\Grass_r(H^1(X,\Q))$. 
Recall also that the Dwyer--Fried sets $\Omega^k_1(X)$ 
are Zariski open, but that $\Omega^k_r(X)$ is not 
necessarily open, if $1<r<b_1(X)$.  We thus 
have the following immediate corollary. 

\begin{corollary} 
\label{cor:not open}
Suppose there is an integer $r\ge 2$ such that $\Omega^k_r(X)$ is 
{\em not} Zariski open in  $\Grass_r(H^1(X,\Q))$.  Then 
$\Sigma^k(X,\Z) \ne  S(\tau^{\R}_1 (\WW^k(X)))^{\compl}$. 
\end{corollary}

\subsection{The straight and formal settings}
\label{subsec:cyclic}

When the space $X$ is locally $k$-straight, 
we may replace in the above the exponential tangent cone to the $k$-th 
characteristic variety of $X$ by the corresponding resonance variety. 

\begin{corollary} 
\label{cor:bns rv locstraight}
Let $X$ be a locally $k$-straight space.  Then, for all $r\ge 1$, 
\[
\Sigma^k(X,\Z) = S(X)\setminus S(\RR^k(X,\R))
\implies 
\Omega^k_r(X) = \Grass_r(H^1(X,\Q)) \setminus \sigma_r (\RR^k(X,\Q)). 
\]
\end{corollary}

\begin{proof}
Follows at once from Theorems \ref{thm:bns omega} and \ref{thm:rat res}.
\end{proof}

Recalling now that every $1$-formal space is locally 
$1$-straight (cf.~Theorem \ref{thm:tcone}), we derive  
the following corollary.

\begin{corollary} 
\label{cor:bns rv 1f}
Let $X$ be a $1$-formal space.  Then, for all $r\ge 1$, 
\[
\Sigma^1(X,\Z) = S(X)\setminus S(\RR^1(X,\R))
\implies 
\Omega^1_r(X) = \Grass_r(H^1(X,\Q)) \setminus  \sigma_r (  \RR^1(X,\Q) ). 
\]
\end{corollary}

Using Corollary \ref{cor:ptau} and Theorem  \ref{thm:tcone}, 
we obtain the following consequences, which partially recover 
Corollary \ref{cor:bns 1f}.

\begin{corollary} 
\label{cor:bns rv}
Let $X$ be a $1$-formal space, and $P\subseteq  H^1(X,\Q)$ 
a linear subspace.  If the unit sphere 
in $P$ is included in $\Sigma^1(X)$, then 
$P \cap \RR^1(X,\Q)  =\{0\}$.
\end{corollary}

\begin{corollary} 
\label{cor:bns 1fg}
Let $G$ be a $1$-formal group, and let $\chi\colon G\to \Z$ 
be a non-zero homomorphism.  If  $\{\pm \chi\} \subseteq \Sigma^1(G)$, 
then $\chi\notin \RR^1(X,\Q)$.
\end{corollary}

The formality assumption is really necessary in the previous two 
corollaries.  For instance, let $X$ be the Heisenberg nilmanifold, i.e., 
the $S^1$-bundle over $S^1\times S^1$ with Euler number $1$.     
Then $\Sigma^1(X)=S(X)=S^1$, yet 
$\RR^1(X,\Q)=H^1(X,\Q)=\Q^2$.  

\subsection{Discussion and examples}
\label{subsec:ex}
As we shall see in the last few sections, there are several 
interesting classes of spaces for which the implication from 
Theorem \ref{thm:bns omega} holds as an equivalence.  
Nevertheless, as the next two examples show, 
neither the implication from Theorem \ref{thm:bns omega},  
nor the one from Corollary \ref{cor:bns rv 1f} can be reversed, 
in general.  We will come back to this point in Example \ref{ex:flag rp2}.

\begin{example}
\label{ex:bs group}
Consider the $1$-relator group 
$G=\langle x_1,x_2\mid x_1^{\,}x_2x_1^{-1}=x_2^2\rangle$. 
Clearly, $G_{\ab}=\Z$, and so $G$ is $1$-formal and 
$\RR^1(G)=\{0\}\subset \C$. A Fox calculus computation 
shows that $\WW^1(G)=\{1,2\}\subset \C^{\times}$; 
thus, $\Omega^1_1(G)=\pt$, and so 
$\Omega^1_1(G)=  \sigma_1(\RR^1(G,\Q))^{\compl}$.
On the other hand, algorithms from \cite{Br87, BR} show that 
$\Sigma^1(G)=\{-1\}$, whereas  
$S(\RR^1(G,\R))^{\compl}=\{\pm 1\}$. 

To see why this is the case, consider the abelianization 
map, $\ab\colon G\surj \Z$. Then $G'=\ker(\ab)$ is isomorphic 
to $\Z[1/2]$.  Hence, $H_1(G',\Q)=\Z[1/2]\otimes \Q=\Q$, which 
explains why the character $\ab$ belongs to $\Omega^1_1(G)$.
On the other hand, the group $\Z[1/2]$ is not finitely generated, 
which explains why $\{\pm \ab\}\not\subseteq \Sigma^1(G)$, 
although $-\ab\in \Sigma^1(G)$. 
\end{example}

\begin{example}  
\label{ex:rp2}
Consider the space $X=S^1\vee \RP^2$, with 
fundamental group $G=\Z*\Z_2$.  As before, 
$X$ is $1$-formal and $\RR^1(X)=\{0\}\subset \C$. 
The maximal free abelian cover $X^{\alpha}$, 
corresponding to the projection $\alpha\colon G\surj \Z$, 
is homotopy equivalent to a countably infinite wedge of 
projective planes.  Thus, $b_1(X^{\alpha})=0$, and so 
$\Omega^1_1(X)=\pt$, which equals $\sigma_1(\RR^1(X,\Q))^{\compl}$. 
On the other hand, since $G$ splits as a non-trivial free product, 
$\Sigma^1(X,\Z)=\emptyset$, which does not equal 
$S(\RR^1(X,\R))^{\compl}=S^0$. 
\end{example}

Finally, here is an example showing how Corollary \ref{cor:not open}
can be used to prove that the inclusions from \eqref{eq:bns bound1} and 
\eqref{eq:sigma xg} are proper, in general.  The construction is based 
on an example of Dwyer and Fried \cite{DF}, as revisited in more detail 
in \cite{Su-conm}. 

\begin{example}
\label{ex:df}
Let $Y = T^3 \vee S^2$.  Then $\pi_1(Y)=\Z^3$, a free abelian 
group on generators $x_1,x_2,x_3$, and $\pi_2 (Y) = \Z{\Z^3}$, 
a free module generated by the inclusion $S^2\inj Y$. 
Attaching a $3$-cell to $Y$ along a map $S^2 \to  Y$ 
representing the element $x_1-x_2+1$ in $\pi_2(Y)$, 
we obtain a CW-complex $X$, with $\pi_1(X) = \Z^3$ and 
$\pi_2(X)=\Z{\Z^3}/(x_1-x_2+1)$.  
Identifying $\widehat{\Z^3}=(\C^{\times})^3$, we have that
$\WW^2(X)=\set{ t \in  (\C^{\times})^3 \mid t_1-t_2+1= 0}$, 
and thus $\tau_1(\WW^2(X))=\{0\}$.   

Making use of Theorem \ref{thm:df cv}, we see that 
$\Omega^2_2(X)$ consists of precisely two points in 
$\Grass_2(\Q^3)=\QP^2$;  in particular, $\Omega^2_2(X)$ 
is not a Zariski open subset.   Corollary \ref{cor:not open} now shows 
that $\Sigma^2(X,\Z)\subsetneqq S(\tau_1(\WW^2(X)))^{\compl}=S^2$.   
On the other hand, $\Sigma^2(\Z^3,\Z)=S^2$; thus, 
$\Sigma^2(X,\Z)\subsetneqq \Sigma^2(\pi_1(X),\Z)$.
\end{example}

\section{Toric complexes}
\label{sect:toric}

In this section, we illustrate our techniques on a class 
of spaces that arise in toric topology, as a basic example 
of polyhedral products. These 
``toric complexes" are both straight and formal, so 
it comes as no surprise that both their $\Omega$-invariants 
and their $\Sigma$-invariants are closely related to the 
resonance varieties.  

\subsection{Toric complexes and right-angled Artin groups}
\label{subsec:toric raag}

Let $L$ be a simplicial complex with $n$ vertices, and let $T^n$
be the $n$-torus, with the standard cell decomposition, and with 
basepoint $*$ at the unique $0$-cell.   

The {\em toric complex}\/ associated to $L$, denoted $T_L$,
is the union of all subcomplexes of the form
$T^{\sigma}=\{x\in T^n \mid \text{$x_i = *$ if $i \notin \sigma$} \}$,
where $\sigma$ runs through the simplices of $L$.
Clearly, $T_L$ is a connected CW-complex; 
its $k$-cells are in one-to-one correspondence 
with the $(k-1)$-simplices of $L$.

Denote by $\sV$ the set of $0$-cells of $L$, and 
by $\sE$ the set of $1$-cells of $L$. The fundamental 
group, $G_L=\pi_1(T_L)$, is the right-angled Artin group associated 
to the graph $\Gamma=(\sV,\sE)$, with a generator $v$ for each vertex  
$v \in \sV$, and a commutation relation $vw=wv$ for each edge 
$\{v, w\} \in \sE$.  

Much is known about toric complexes and their fundamental 
groups.  For instance, the group $G_L$ has as classifying space 
the toric complex $T_{\Delta}$, where $\Delta=\Delta_{L}$ 
is the flag complex of $L$, i.e., the maximal simplicial 
complex with $1$-skeleton equal to that of $L$.  
Moreover, the homology groups of $T_L$ are torsion-free, 
while the cohomology ring $H^*(T_L,\Z)$ is 
isomorphic to the exterior Stanley-Reisner ring of $L$, 
with generators the dual classes $v^*\in H^1(T_L,\Z)$, 
and relations the 
monomials corresponding to the missing faces of $L$.  
Finally, all toric complexes are formal~spaces.

For more details and references on all this, we refer to 
\cite{PS-mathann, PS-adv, PS-plms, Su-aspm}.

\subsection{Jump loci and $\Omega$-invariants}
\label{subsec:jump omega}
The resonance and characteristic varieties of right-angled 
Artin groups and toric complexes were studied in 
\cite{PS-mathann} and \cite{DPS-duke}, with the 
complete computation achieved in \cite{PS-adv}. 
We recall here those results, in a form suited 
for our purposes. 

Let $\k$ be a coefficient field.  Fixing an ordering 
on the vertex set $\sV$ allows us to 
identify $H^1(T_L,\k^{\times})$ with the algebraic 
torus $(\k^{\times})^{\sV}=(\k^{\times})^n$ and 
$H^1(T_L,\k)$ with the vector space $\k^{\sV}=\k^n$.  
Each subset $\sW \subseteq \sV$ gives rise to 
an algebraic subtorus 
$(\k^{\times})^{\sW} \subset (\k^{\times})^{\sV}$  
and a coordinate subspace $\k^{\sW} \subset \k^{\sV}$.

In what follows, we denote by $L_{\sW}$ the subcomplex 
induced by $L$ on $\sW$, and by $\lk_K(\sigma)$ the link 
of a simplex $\sigma\in L$ in a subcomplex $K \subseteq L$. 

\begin{theorem}[\cite{PS-adv}]
\label{thm:jump toric}
Let $L$ be a simplicial complex on vertex set $\sV$.  
Then, for all $i\ge 1$, 
\begin{equation}
\label{eq:vtl}
\VV^i(T_L,\k)= \bigcup_{\sW} \, (\k^{\times})^{\sW} 
\quad \text{and}  \quad
\RR^i(T_L,\k)=\bigcup_{\sW}\: \k^{\sW},
\end{equation}
where, in both cases, the union is taken over all subsets 
$\sW \subseteq \sV$ for which there is a simplex 
$\sigma \in L_{\sV\setminus \sW}$ and an index $j\le i$ such that 
$\widetilde{H}_{j-1-\abs{\sigma}}(\lk_{L_{\sW}}(\sigma), \k)\ne 0$.
\end{theorem}

In degree $1$, the resonance formula takes a simpler form, 
already noted in \cite{PS-mathann}. Clearly, $\RR^1(T_{L},\k)$ 
depends only on the $1$-skeleton $\Gamma=L^{(1)}$; 
moreover, $\RR^1(T_{L},\k) = \bigcup_{\sW}  \k^{\sW}$, 
where the union is taken over all maximal subsets 
$\sW\subseteq\sV$ for which the induced graph
$\G_{\sW}$ is disconnected. 

As a consequence of Theorem \ref{thm:jump toric}, we 
see that every toric complex is a straight space. 
Theorem \ref{thm:df straight}\eqref{s2}, then, allows us to 
determine the Dwyer--Fried invariants of such spaces.

\begin{corollary}[\cite{PS-plms}, \cite{Su-aspm}]
\label{cor:df toric}
Let $L$ be a simplicial complex on vertex set $\sV$. 
Then, for all $i, r\ge 1$, 
\begin{equation}
\label{eq:omegatl}
\Omega^i_r(T_L) = \Grass_r(\Q^{\sV})\setminus  
 \sigma_r(\RR^i(T_L,\Q)). 
\end{equation}
\end{corollary}

\subsection{$\Sigma$-invariants}
\label{subsec:raag}

In \cite{BB}, Bestvina and Brady considered 
the ``diagonal" homomorphism $\nu\colon G_L \to \Z$, 
$v \mapsto 1$, and the finiteness properties of the 
corresponding subgroup, $N_L=\ker(\nu)$. 
One of the main results of \cite{BB} determines the maximal 
integer $k$ for which $\nu$ belongs to $\Sigma^k(G_L,\Z)$. 

The picture was completed by Meier, Meinert, and 
VanWyk \cite{MMV} and by Bux and Gonzalez \cite{BG}, 
who computed explicitly the Bieri--Neumann--Strebel--Renz 
invariants of right-angled Artin groups.  

\begin{theorem}[\cite{MMV, BG}]
\label{thm:mmv bg}
Let $L$ be a simplicial complex, and let $\Delta$ be 
the associated flag complex. 
Let  $\chi\in S(G_L)$ be a non-zero homomorphism, 
with support $\sW=\set{v \in \sV \mid \chi(v)\ne 0}$, 
and let $\k=\Z$ or a field.  
Then, $\chi\in \Sigma^{i}(G_{L},\k)$ if and only if
\begin{equation}
\label{eq:hlkdelta}
\widetilde{H}_{j}(\lk_{\Delta_{\sW}} (\sigma),\k)=0, 
\end{equation}
for all $\sigma\in \Delta_{\sV\setminus \sW}$ 
and $-1\le j\le i-\dim(\sigma)-2$.
\end{theorem}

We would like now to compare the BNSR invariants 
of a toric complex $T_L$ to the resonance varieties 
of $T_L$.  Using the fact that toric complexes 
are straight spaces, Corollary \ref{cor:bns straight} gives 
\begin{equation}
\label{eq:bns res tc}
\Sigma^i(T_L,\Z)\subseteq S(T_L)\setminus S(\RR^i(T_L,\R)).
\end{equation}

For right-angled Artin groups, we can say more.  
Comparing the description of the $\Sigma$-invariants 
of the group $G_L$ given in Theorem \ref{thm:mmv bg} 
to that of the resonance varieties of the space 
$T_{\Delta}=K(G_L,1)$  given in Theorem \ref{thm:jump toric}, 
yields the following result. 

\begin{corollary}[\cite{PS-plms}]
\label{cor:res bns raag}
Let $G_L$ be a right-angled Artin group.
For each $i\ge 0$, the following hold. 
\begin{enumerate}
\item \label{bns1}
$\Sigma^i(G_{L},\R) = S(\RR^i(G_{L},\R))^{\compl}$. 
\item \label{bns2}
$\Sigma^i(G_{L}, \Z)= S(\RR^i(G_{L}, \R))^{\compl}$, 
provided that, for every $\sigma \in \Delta$, and every 
$\sW\subseteq \sV$ with $\sigma \cap W=\emptyset$, 
the groups $\widetilde{H}_j(\lk_{\Delta_{\sW}}(\sigma),\Z)$ are  
torsion-free, for all  $j\le i-\dim(\sigma)-2$. 
\end{enumerate}
\end{corollary}

The torsion-freeness condition from 
Corollary \ref{cor:res bns raag}\eqref{bns2} is always 
satisfied in degree $i=1$.  Thus,
\begin{equation}
\label{eq:sigma1res1}
\Sigma^1(G_L, \Z) =  S(\RR^1(G_{L}, \R))^{\compl},
\end{equation}
an equality already proved (by different methods) 
in \cite{PS-mathann}.
Nevertheless, the condition is not always satisfied 
in higher degrees, thus leading to situations where 
the equality from Corollary \ref{cor:res bns raag}\eqref{bns2} 
fails.  The next example (extracted from \cite{PS-plms}) 
illustrates this phenomenon, while also showing that the 
implication from Theorem \ref{thm:bns omega} cannot 
always be reversed, even for right-angled Artin groups. 

\begin{example}
\label{ex:flag rp2}
Let $\Delta$ be a flag triangulation of the real projective plane, 
$\RP^2$, and let $\nu\colon G_{\Delta}\to \Z$ be the diagonal 
homomorphism.  
Then $\nu\not\in \Sigma^2( G_{\Delta}, \Z)$, even though 
$\nu \in \Sigma^2( G_{\Delta}, \R)$. 
Consequently, 
\begin{equation}
\label{eq:flag rp2}
\Sigma^2( G_{\Delta}, \Z) \subsetneqq 
S(\RR^2(G_{\Delta}, \R))^{\compl},
\end{equation}
although $\Omega^2_r(G_{\Delta}) = 
\sigma_r(\RR^2(T_{\Delta},\Q))^{\compl}$, 
for all $r\ge 1$.  
\end{example}

\section{Quasi-projective varieties}
\label{sec:qproj}

We now discuss the cohomology jumping loci, the Dwyer--Fried 
invariants, and the Bieri--Neumann--Strebel--Renz invariants of 
smooth, complex projective and quasi-projective varieties. 

\subsection{Complex algebraic varieties}
\label{subsec:qproj}

A smooth, connected manifold $X$ is said to be a 
{\em (smooth) quasi-projective variety}\/ 
if there is a smooth, complex projective variety $\overline{X}$ and a 
normal-crossings divisor $D$ such that $X=\overline{X}\setminus D$. 
By a well-known result of Deligne, each cohomology group of a 
quasi-projective variety $X$ admits a mixed Hodge structure.   
This puts definite constraints on the topology of such varieties.   

For instance, if $X$ admits a non-singular compactification 
$\overline{X}$ with $b_1(\overline{X})=0$, the weight 
$1$ filtration on $H^1(X,\C)$ vanishes; in turn, by work 
of Morgan, this implies the $1$-formality of $X$.   
Thus, as noted by Kohno, if $X$ is the complement 
of a hypersurface in $\CP^n$, then $\pi_1(X)$ is $1$-formal.  
In general, though, quasi-projective varieties need not 
be $1$-formal.  

If $X$ is actually a (compact, smooth) projective variety, 
then a stronger statement holds:  as shown by Deligne, 
Griffiths, Morgan, and Sullivan, such a manifold 
(and, more generally, a compact K\"ahler manifold) 
is formal. In general, though, quasi-projective varieties 
are not formal, even if their fundamental groups are 
$1$-formal. 

\begin{example}
\label{ex:qp1f}
Let $T = E^{n}$ be the $n$-fold product of an elliptic 
curve $E=\C/\Z\oplus\Z$.  The closed form 
$\frac{1}{2} \sqrt{-1} \sum_{i=1}^n dz_i \wedge d\bar{z}_i = 
\sum_{i=1}^n dx_i \wedge dy_i$ defines a cohomology 
class $\omega$ in $H^{1,1}(T)\cap H^2 (T, \Z)$.  
By the Lefschetz theorem on $(1,1)$-classes 
(see~\cite[p.~163]{GH78}), $\omega$ can be 
realized as the first Chern class of an algebraic 
line bundle $L \to T$.   

Let $X$ be the complement of the zero-section of $L$; 
then $X$ is a connected, smooth, quasi-projective variety.  
Moreover, $X$ deform-retracts onto $N$, the total space of 
the circle bundle over the torus $T=(S^1)^{2n}$ with Euler 
class $\omega$.  Clearly, the Heisenberg-type nilmanifold 
$N$ is not a torus, and thus it is not formal.  In fact, as shown 
by M\u{a}cinic in \cite[Remark 5.4]{Mc}, the manifold $N$ is 
$(n-1)$-formal, but not $n$-formal.  Thus, if $n=1$, the variety 
$X$ is not $1$-formal, whereas if $n>1$, the variety $X$ is 
$1$-formal, but not formal. 
\end{example}

\subsection{Cohomology jump loci}
\label{subsec:proj jumps}
The existence of mixed Hodge structures on the cohomology 
groups of connected, smooth, complex quasi-projective varieties 
also puts definite constraints on the nature of their cohomology 
support loci.  

The structure of the characteristic varieties of such 
spaces (and, more generally, K\"{a}hler and 
quasi-K\"{a}hler manifolds) was determined 
through the work of Beauville,  Green and 
Lazarsfeld, Simpson, Campana, and 
Arapura in the 1990s.  Further improvements 
and refinements have come through the recent 
work of Budur, Libgober, Dimca, Artal-Bartolo, Cogolludo, 
and Matei.  We summarize these results,  
essentially in the form proved by Arapura, but  
in the simplified (and slightly updated) form we 
need them here.

\begin{theorem}[\cite{Ar}]
\label{thm:arapura}
Let $X=\overline{X}\setminus D$, where $\overline{X}$ 
is a smooth, projective variety and $D$ is a normal-crossings 
divisor. 
\begin{enumerate}
\item \label{a1}
If either $D=\emptyset$ or $b_1(\overline{X})=0$, 
then each characteristic variety $\VV^i(X)$ 
is a finite union of unitary translates of algebraic 
subtori of $H^1(X,\C^{\times})$. 
\item \label{a2}
In degree $i=1$, the condition that 
$b_1(\overline{X})=0$ if $D\ne \emptyset$ may be lifted.  
Furthermore, each positive-dimensional component 
of $\VV^1(X)$ is of the form $\rho \cdot T$, 
with $T$ an algebraic subtorus, and 
$\rho$ a {\em torsion}\/ character. 
\end{enumerate}
\end{theorem}

For instance, if $C$ is a connected, smooth complex curve 
with $\chi(C)<0$, then $\VV^1(C)=H^1(C,\C^{\times})$.  
More generally, if $X$ is a smooth, quasi-projective variety, 
then every positive-dimensional component of $\VV^1(X)$ 
arises by pullback along a suitable pencil.   
More precisely, if $\rho \cdot T$ is such a component, 
then $T=f^*(H^1(C,\C^{\times}))$, for some curve $C$, 
and some holomorphic, surjective map 
$f\colon X \to C$ with connected generic fiber.  

In the presence of $1$-formality, the quasi-projectivity  
of $X$ also imposes stringent conditions on the degree~$1$ 
resonance varieties. Theorems \ref{thm:arapura} and \ref{thm:tcone} 
yield the following characterization of these varieties. 

\begin{corollary}[\cite{DPS-duke}]
\label{cor:res qp} 
Let $X$ be a $1$-formal, smooth, quasi-projective variety. 
Then $\RR^1(X)$ is a finite union of rationally defined 
linear subspaces of $H^1(X,\C)$. 
\end{corollary}

In fact, much more is proved in \cite{DPS-duke} about those 
subspaces. For instance, any two of them intersect 
only at $0$, and the restriction of the cup-product map 
$H^1(X,\C)\wedge H^1(X,\C) \to H^2(X,\C)$ to any 
one of them has rank equal to either $0$ or $1$. 

\begin{corollary}[\cite{Su-aspm}]
\label{cor:straight qp}
If $X$ is a $1$-formal, smooth, quasi-projective 
variety, then $X$ is locally $1$-straight.  Moreover, 
$X$ is $1$-straight if and only if $\WW^1(X)$ contains no 
positive-dimensional translated subtori.
\end{corollary}

\subsection{$\Omega$-invariants}
\label{subsec:proj omega}

The aforementioned structural results regarding the 
cohomology jump loci of smooth, quasi-projective varieties 
inform on the Dwyer--Fried sets of such  varieties.  For instance, 
Theorem \ref{thm:arapura} together with  Proposition 
\ref{prop:tau schubert} yield the following corollary. 

\begin{corollary}
\label{cor:df quasiproj}
Let $X=\overline{X}\setminus D$ be a smooth, quasi-projective 
variety with $D=\emptyset$ or $b_1(\overline{X})=0$. 
If $\WW^i(X)$ contains no positive-dimensional 
translated subtori, then 
$\Omega^i_r(X) = \sigma_r(\tau^{\Q}_1(\WW^i(X)))^{\compl}$, 
for all $r\ge 1$. 
\end{corollary}

Likewise, Corollary \ref{cor:straight qp} together with 
Theorem \ref{thm:df straight} yield the following corollary. 

\begin{corollary}[\cite{Su-aspm}]
\label{cor:df 1fqp}
Let $X$ be a $1$-formal, smooth, quasi-projective variety.  
Then:
\begin{enumerate}
\item \label{qk1}
$\Omega^1_1(X) =\bP(\RR^1(X,\Q))^{\compl}$ 
and 
$\Omega^1_r(X) \subseteq \sigma_r(\RR^1(X,\Q))^{\compl}$, 
for $r\ge 2$. 
\item \label{qk2}
If $\WW^1(X)$ contains no positive-dimensional 
translated subtori, then 
$\Omega^1_r(X) = \sigma_r(\RR^1(X,\Q))^{\compl}$, 
for all $r\ge 1$. 
\end{enumerate} 
\end{corollary}

If the characteristic variety $\WW^1(X)$ contains positive-dimensional 
translated components, the resonance variety $\RR^1(X,\Q)$ 
may fail to determine all the Dwyer--Fried sets $\Omega^1_r(X)$.  
This phenomenon is made concrete by the following result. 

\begin{theorem}[\cite{Su-aspm}]
\label{thm:tt}
Let $X$ be a $1$-formal, smooth, quasi-projective variety. 
Suppose  $\WW^1(X)$ has a $1$-dimensional component 
not passing through $1$, while 
$\RR^1(X)$ has no codimension-$1$ components.
Then $\Omega^1_{2}(X)$ is 
strictly contained in $\Grass_2(H^1(X,\Q))\setminus 
\sigma_{2}(\RR^1(X,\Q))$. 
\end{theorem}

\subsection{$\Sigma$-invariants}
\label{subsec:proj sigma}
The cohomology jump loci of smooth, quasi-projective varieties 
also inform on the Bieri--Neumann--Strebel sets of such varieties. 
For instance, using Corollary \ref{cor:bns rv 1f}, we may identify 
a class of $1$-formal, quasi-projective varieties for which  
inclusion \eqref{eq:bns res bound} is strict.

\begin{corollary}
\label{cor:tt}
Let $X$ be a $1$-formal, smooth, quasi-projective variety. 
Suppose  $\WW^1(X)$ has a $1$-dimensional component 
not passing through $1$, while 
$\RR^1(X)$ has no codimension-$1$ components.
Then $\Sigma^1(X)$ is 
strictly contained in $S(X) \setminus S(\RR^1(X,\R))$. 
\end{corollary}

We shall see in Example \ref{ex:deleted B3} a concrete variety 
to which this corollary applies. 

In the case of smooth, complex projective varieties (or, more 
generally, compact K\"{a}hler manifolds), a different approach 
is needed in order to show that inclusion \eqref{eq:bns res bound} 
may be strict.  Indeed, by Theorem \ref{thm:arapura}, all 
components of $\WW^1(\overline{X})$ are even-dimensional, 
so Corollary \ref{cor:tt} does not apply. 

On the other hand, as shown by Delzant in \cite{De10}, the 
BNS invariant of a compact K\"{a}hler manifold $M$ is 
determined by the pencils supported by $M$.    

\begin{theorem}[\cite{De10}] 
\label{thm:delzant}
Let $M$ be a compact K\"{a}hler manifold.   
Then 
\begin{equation}
\label{eq:sigma1m}
\Sigma^1(M)= S(M) \setminus\bigcup\nolimits_{\alpha} S( f_{\alpha}^* 
( H^1(C_{\alpha}, \R)) ),
\end{equation}
where the union is taken 
over those pencils $f_{\alpha}\colon M\to C_{\alpha}$
with the property that either $\chi(C_{\alpha})<0$, or 
$\chi(C_{\alpha})=0$ and $f_{\alpha}$ has some multiple fiber. 
\end{theorem}

This theorem, together with results from \cite{DPS-duke}, 
yields the following characterization of those compact 
K\"{a}hler manifolds $M$ for which the inclusion from 
Corollary \ref{cor:bns 1f} holds as equality.  

\begin{theorem}[\cite{PS-plms}]
\label{thm:kahler}
Let  $M$ be a compact K\"{a}hler manifold. 
Then $\Sigma^1(M)=S(\RR^1(M, \R))^{\compl}$ 
if and only if there is no pencil $f\colon M\to E$ onto an elliptic 
curve $E$ such that $f$ has multiple fibers.  
\end{theorem}

\subsection{Examples and discussion}
\label{subsec:exdis}
As noted in \cite[Remark 16.6]{PS-plms}, a general construction 
due to Beauville  \cite{Be} shows that equality does not always 
hold in Theorem \ref{thm:kahler}.  More precisely, if $N$ is a 
compact K\"{a}hler manifold on which a finite group $\pi$ acts 
freely, and  $p\colon C \to E$ is a ramified, regular $\pi$-cover 
over an elliptic curve, with at least one ramification point, 
then the quotient $M=(C \times N)/\pi$ is 
a compact K\"{a}hler manifold admitting a pencil 
$f\colon M\to E$ with multiple fibers.  

An example of this construction is given in \cite{DP}.    
Let $C$ be a Fermat quartic in $\CP^2$, viewed as a $2$-fold branched 
cover of $E$, and let $N$ be a simply-connected compact K\"{a}hler 
manifold admitting a fixed-point free involution, for instance, a 
Fermat quartic in $\CP^3$, viewed as a $2$-fold unramified 
cover of the Enriques surface.  
Then, the  K\"{a}hler manifold $M=(C\times N)/\Z_2$  
admits a pencil with base $E$, having four multiple 
fibers, each of multiplicity $2$; thus, $\Sigma^1(M,\Z)=\emptyset$. 
Moreover, direct computation shows that $\RR^1(M)=\{0\}$, and so  
$\Sigma^1(M,\Z) \subsetneqq S(\RR^1(M, \R))^{\compl}$. 

We provide here another example, in the lowest possible dimension, 
using a complex surface studied by Catanese, Ciliberto, and 
Mendes Lopes in \cite{CCM}.  
For this manifold $M$, the resonance variety does not vanish, 
yet $\Sigma^1(M,\Z)$ is strictly contained in the complement of 
$S(\RR^1(M, \R))$.  We give an alternate explanation of this fact 
which is independent of Theorems \ref{thm:delzant} and \ref{thm:kahler}, 
but relies instead on results from \cite{Su11} and on Corollary \ref{cor:bns rv 1f}.

\begin{example}
\label{ex:omega kahler}
Let $C_1$ be a (smooth, complex) curve of genus $2$ 
with an elliptic involution $\sigma_1$ and let $C_2$ be 
a curve of genus $3$ with a free involution $\sigma_2$. 
Then $\Sigma_1=C_1/\sigma_1$ is a curve of genus $1$, 
and $\Sigma_2=C_2/\sigma_2$ is a curve of genus $2$. 

Now let $M=(C_1 \times C_2)/\Z_2$, 
where $\Z_2$ acts via the involution $\sigma_1 \times \sigma_2$.  
Then $M$ is a smooth, complex projective surface.  
Projection onto the first coordinate yields a pencil 
$f_1 \colon M\to \Sigma_1$ with two multiple fibers, 
each of multiplicity $2$, while 
projection onto the second coordinate 
defines a smooth fibration $f_2\colon M\to \Sigma_2$. 
By Theorem \ref{thm:kahler}, we have that 
$\Sigma^1(M,\Z) \subsetneqq S(\RR^1(M, \R))^{\compl}$.

Here is an alternate explanation. Using the fact that $H_1(M,\Z)=\Z^6$, 
we may identify $H^1(M,\C^{\times}) = (\C^{\times})^6$.   In  \cite{Su11}, 
we showed that   
\begin{equation}
\label{eq:v1m}
\VV^1(M)=
\{t_4=t_5=t_6=1, \, t_3=-1\}\cup 
\{ t_1=t_2=1\}, 
\end{equation}
with the two components corresponding to the pencils 
$f_1$ and $f_2$, respectively, from which we inferred 
that the set $\Omega^1_2(M)$ is not open, 
not even in the usual topology on $\Grass_2(\Q^6)$. 
Now, from \eqref{eq:v1m}, we also see that $\RR^1(M)=\{ x_1=x_2=0\}$. 
Since $\Omega^1_2(M)$ is not open, it must be a proper subset 
of $\sigma_2( \RR^1(M,\Q) )^{\compl}$.  
In view of Corollary \ref{cor:bns rv 1f}, we conclude once again that 
$\Sigma^1(M,\Z) \subsetneqq S(\RR^1(M, \R))^{\compl}$. 
\end{example}

\section{Configuration spaces}
\label{sect:config}

We now consider in more detail a particularly interesting class 
of quasi-projective varieties, obtained by deleting 
the ``fat diagonal" from the $n$-fold Cartesian product of 
a smooth, complex algebraic curve. 

\subsection{Ordered configurations on algebraic varieties}
\label{subsec:fxn}

A construction due to Fadell and Neuwirth associates 
to a space $X$ and a positive integer $n$ the space of 
ordered configurations of $n$ points in $X$, 
\begin{equation}
\label{eq:conf}
F(X,n) = \{ (x_1, \dots , x_n) \in X^{n} 
\mid x_i \ne x_j \text{ for } i\ne j\}.
\end{equation}

The most basic example is the configuration space 
of $n$ ordered points in $\C$, which is a classifying 
space for $P_n$, the pure braid group on $n$ strings, 
whose cohomology ring was computed by Arnol'd in the 
late 1960s.  

The $E_2$-term of the Leray spectral sequence for 
the inclusion $F(X,n)\inj X^n$ was described concretely 
by Cohen and Taylor in the late 1970s. 
If $X$ is a smooth, complex projective variety 
of dimension $m$, then, as shown by Totaro in \cite{To96}, 
the Cohen--Taylor spectral sequence collapses at the 
$E_{m+1}$-term, and $H^*(F(X,n),\C)=E_{m+1}$, 
as graded algebras.  

Particularly interesting is the case of a Riemann surface 
$\Sigma_g$.  The ordered configuration space $F(\Sigma_g,n)$ 
is a classifying space for $P_{n}(\Sigma_g)$, the pure braid group 
on $n$ strings of the underlying surface.   
In \cite{Be94}, Bezrukavnikov gave explicit presentations for the 
Malcev Lie algebras $\m_{g,n}=\m(P_{n}(\Sigma_g))$, from which 
he concluded that the pure braid groups on surfaces are $1$-formal 
for $g>1$ or $g=1$ and $n\le 2$, but not $1$-formal for $g=1$ 
and $n\ge 3$.   The non-$1$-formality of the groups 
$P_{n}(\Sigma_1)$, $n\ge 3$, is also established  
in \cite{DPS-duke}, by showing that the tangent cone 
formula  \eqref{eq:tcone} fails in this situation (see 
Example \ref{ex:conf spaces} below for the case $n=3$).

\begin{remark}
\label{rem:cee}
In \cite[Proposition 5]{CEE}, Calaque, Enriquez, and Etingof prove 
that $P_{n}(\Sigma_1)$ is formal, for all $n\ge 1$.  But the notion 
of formality that these authors use is weaker than the usual notion 
of $1$-formality:  their result is that $\m_{1,n}$ is isomorphic 
as a filtered Lie algebra with the completion (with respect to the 
bracket length filtration) of the associated graded Lie algebra, 
$\gr(\m_{1,n})$.  The failure of $1$-formality comes from the 
fact that $\gr(\m_{1,n})$ is {\em not}\/ a quadratic Lie algebra, 
for $n\ge 3$.
\end{remark}

\subsection{Ordered configurations on the torus}
\label{subsec:ft}
For the reasons outlined above, it makes sense 
to look more carefully at the configuration  
spaces of an elliptic curve $\Sigma_1$.  
The resonance varieties $\RR^1(F(\Sigma_1,n))$ were computed 
in \cite{DPS-duke}, while the positive-dimensional components 
of  $\VV^1(F(\Sigma_1,n))$ were determined  in \cite{Di10}. 

Since $\Sigma_1=S^1\times S^1$ is a topological group, the  
space $F(\Sigma_1,n)$ splits up to homeomorphism as a direct product, 
$F(\Sigma'_1, n-1)\times \Sigma_1$, where $\Sigma'_1$ 
denotes $\Sigma_1$ with the identity removed.  
Thus, for all practical purposes, it 
is enough to consider the space $F(\Sigma'_1, n-1)$. 
For the sake of concreteness, we will work out in 
detail the case $n=3$; the general case may be treated similarly.  

\begin{example}
\label{ex:conf spaces}
Let $X=F(\Sigma'_1,2)$ be the configuration space of $2$ 
labeled points on a punctured torus.  
The cohomology ring of $X$ is the exterior algebra on 
generators $a_1,a_2, b_1,b_2$ in degree $1$, modulo 
the ideal spanned by the forms 
$a_1 b_2+a_2 b_1$, $a_1 b_1$, and $a_2 b_2$. 
The first resonance variety is an 
irreducible quadric hypersurface in $\C^4$, given by 
\[
\RR^1(X)=\{x_1y_2-x_2y_1=0\}.
\]
Corollary \ref{cor:res qp}, then, shows that $X$ is not $1$-formal. 

The first characteristic variety of $X$ consists of three 
$2$-dimensional algebraic subtori of $(\C^{\times})^4$:
\[
\VV^1(X)=\{ t_1=t_2=1\} \cup \{s_1=s_2=1\} \cup 
\{ t_1s_1=t_2s_2=1\}.  
\]
These three subtori arise from the fibrations 
$F(\Sigma'_1,2) \to \Sigma'_1$ 
obtained by sending a point $(z_1,z_2)$ to $z_1$, $z_2$, 
and $z_1z_2^{-1}$, respectively.  It follows that 
$\tau_1(\VV^1(X))=\TC_1(\VV^1(X))$, but both  
types of tangent cones are properly contained in 
the resonance variety $\RR^1(X)$.  Moreover, 
the characteristic subspace arrangement $\CC_1(X)$ 
consists of three, pairwise transverse planes in $\Q^4$, 
namely, 
\[
L_1=\{ x_1=x_2=0\}, \ L_2 =\{y_1=y_2=0\}, 
\ L_3=\{ x_1+y_1=x_2+y_2=0\}.
\]

By Proposition \ref{prop:tau schubert}, the Dwyer--Fried sets  
$\Omega^1_r(X)$ are obtained by removing from $\Grass_r(\Q^4)$  
the Schubert varieties $\sigma_r(L_1)$, $\sigma_r(L_2)$, and 
$\sigma_r(L_3)$. We treat each rank $r$ separately.

\begin{itemize}
\item
 When $r=1$, the set  
$\Omega^1_1(X)$ is the complement in $\QP^3$ 
of the three projective lines defined by $L_1$, $L_2$, and $L_3$.  
\\[-8pt]
\item
When $r=2$, the Grassmannian $\Grass_2(\Q^4)$ 
is the quadric hypersurface in $\QP^5$ given in Pl\"{u}cker 
coordinates by the equation
$p_{12}p_{34}-p_{13}p_{24}+p_{23}p_{14}=0$.
The set $\Omega^1_2(X)$, then, is the complement 
in $\Grass_2(\Q^4)$ of the variety cut out by the 
hyperplanes $p_{12}=0$, $p_{34}=0$, and  
$p_{12}-p_{23}+p_{14}+p_{34}=0$.
\\[-8pt]
\item
When $r\ge 3$, the set $\Omega^1_r(X)$ is empty.
\end{itemize}

Finally, by Corollary \ref{cor:bns tau}, the BNS set 
$\Sigma^1(X,\Z)$ is included in the complement in $S^3$ of 
the three great circles cut out by the real planes spanned 
by $L_1$, $L_2$, and $L_3$, respectively.  It would be 
interesting to know whether this inclusion is actually an equality. 
\end{example}

\section{Hyperplane arrangements}
\label{sect:arrs}

We conclude with another interesting class 
of quasi-projective varieties, obtained by deleting 
finitely many hyperplanes from a complex affine 
space.

\subsection{Complement and intersection lattice}
\label{subsec:arr stuff}

A (central) hyperplane arrangement $\A$ is a finite collection 
of codimension $1$ linear subspaces in a complex affine 
space $\C^{\ell}$.   A defining polynomial for 
$\A$ is the product $Q(\A)=\prod_{H\in \A} \alpha_H$, where 
$\alpha_H\colon \C^{\ell} \to \C$ is a linear form whose kernel 
is $H$.

The main topological object associated 
to an arrangement is its complement, 
$X(\A)=\C^{\ell}\setminus\bigcup_{H\in \A}H$. 
This is a connected, smooth, quasi-projective variety, 
whose homotopy-type invariants are intimately tied  
to the combinatorics of the arrangement.  The latter 
is encoded in the intersection lattice, $L(\A)$, which 
is the poset of all (non-empty) intersections of 
$\A$, ordered by reverse inclusion. The rank 
of the arrangement, denoted $\rk( \A )$, is the 
codimension of $\bigcap_{H\in \A} H$. 

\begin{example}
\label{ex:br}
A familiar example is the rank $\ell-1$ braid arrangement, 
consisting of the diagonal hyperplanes $H_{ij}=\{z_i-z_j=0\}$ 
in $\C^{\ell}$. The complement is the configuration 
space $F(\C,\ell)$, while the intersection lattice is the lattice 
of partitions of $\set{1,\dots,\ell}$, ordered by refinement.  
\end{example}

For a general arrangement $\A$, the cohomology ring of the 
complement was computed by Brieskorn in the early 1970s, 
building on work of Arnol'd on the cohomology ring of the braid 
arrangement.  It follows from Brieskorn's work that the space 
$X(\A)$ is formal.  In 1980, Orlik and Solomon gave a simple 
combinatorial description of the ring $H^*(X(\A),\Z)$:  it is the 
quotient of the exterior algebra on degree-one classes $e_H$ 
dual to the meridians around the hyperplanes $H\in \A$, 
modulo a certain ideal (generated in degrees greater than one) 
determined by the intersection lattice. 

Let $\overline{\A}=\{ \bP(H)\}_{H\in \A}$ be the 
projectivization of $\A$, and let 
$X(\overline{\A})$ be its complement in $\CP^{\ell-1}$. 
The standard $\C^{\times}$-action on $\C^{\ell}$ restricts 
to a free action on $X(\A)$; the resulting 
fiber bundle, $\C^{\times} \to X(\A) \to X(\overline{\A})$,  
is readily seen to be trivial. Under the resulting identification,  
$X(\A)= X(\overline{\A}) \times \C^{\times}$, the 
group $H^1(\C^{\times},\Z)=\Z$ is spanned by the vector  
$\sum_{H\in \A} e_H\in H^1(X(\A),\Z)$.

\subsection{Cohomology jump loci}
\label{subsec:cjl arr}

The resonance varieties $\RR^i(\A):=\RR^i(X(\A),\C)$ 
were first defined and studied by Falk in \cite{Fa97}.  
Clearly, these varieties depend only on the graded 
ring $H^*(X(\A),\C)$, and thus, only on the intersection 
lattice $L(\A)$.  

Now fix a linear ordering on the hyperplanes of $\A$, and 
identify $H^1(X(\A),\C)=\C^n$, where $n=\abs{\A}$.  
From the product formula \eqref{eq:rprod} for resonance 
varieties (or from an old result of Yuzvinsky \cite{Yu95}), 
we see that $\RR^i(\A)$ is isomorphic to $\RR^i(\overline{\A})$, 
and lies in the hyperplane $x_1+\cdots +x_n=0$ inside 
$\C^n$.  
Similarly, the characteristic varieties $\VV^i(\A):=\VV^i(X(\A),\C)$ 
lie in the subtorus $t_1\cdots t_n=1$ inside the complex 
algebraic torus $H^1(X(\A),\C^{\times})=(\C^{\times})^n$. 

In view of the Lefschetz-type theorem of Hamm and L\^{e}, 
taking a generic two-dimen\-sional section does not 
change the fundamental group of the complement. 
Thus, in order to describe the variety $\RR^1(\A)$, 
we may assume $\A$ is a affine arrangement of $n$ lines 
in $\C^2$, for which no two lines are parallel.  

The structure of the first resonance variety of an arrangement 
was worked out in great detail in work of  Cohen, Denham, Falk, 
Libgober, Pereira, Suciu, Yuzvinsky, and many others.  It is known 
that each component of $\RR^1(\A)$ is a linear subspace in $\C^n$, 
while any  two distinct components meet only at $0$. The simplest 
components of  the resonance variety are those corresponding to multiple 
points  of $\A$:  if $m$ lines meet at a point, then $\RR^1(\A)$ 
acquires an $(m-1)$-dimensional linear subspace.  The remaining 
components (of dimension either $2$ or $3$), correspond to 
certain ``neighborly partitions" of sub-arrange\-ments of $\A$. 

\begin{example}
\label{ex:braid arr}
Let $\A$ be a generic $3$-slice of the braid 
arrangement of rank $3$, with defining polynomial 
$Q(\A)=z_0z_1z_2(z_0-z_1)(z_0-z_2)(z_1-z_2)$. 
Take a generic plane section, and label the 
corresponding lines as $1,\dots ,6$. 
Then, the variety $\RR^{1}(\A)\subset \C^6$ has $4$ 
local components, corresponding to the triple 
points $124, 135, 236, 456$, and one non-local 
component, corresponding to the neighborly 
partition $(16| 25 | 34)$.  
\end{example}

From Theorem \ref{thm:arapura}, we know that 
$\VV^1(\A)$ consists of subtori in 
$(\C^{\times})^n$, possibly translated by roots of unity, 
together with a finite number of torsion points.   
By Theorem \ref{thm:tcone}---first proved in the context 
of hyperplane arrangements by Cohen--Suciu and Libgober---%
we have that $\TC_1(\VV^1(\A))=\RR^1(\A)$.  Thus, the components 
of $\VV^1(\A)$ passing through the origin are completely determined 
by $\RR^1(\A)$, and hence, by $L(\A)$:  
to each linear subspace $L$ in $\RR^1(\A)$ 
there corresponds an algebraic subtorus, $T=\exp(L)$,  
in $\VV^1(\A)$.  

As pointed out in \cite{Su02}, though, the characteristic variety 
$\VV^1(\A)$ may contain translated subtori---that is, components 
not passing through $1$.  Despite much work since then, it is still 
not known whether such components are combinatorially determined.

\subsection[Upper bounds for the Omega- and Sigma-sets]
{Upper bounds for the $\Omega$- and $\Sigma$-sets}
\label{subsec:omsig arr}
We are now ready to consider the Dwyer--Fried and 
the Bieri--Neumann--Strebel--Renz invariants associated 
to a hyperplane arrangement $\A$.  For simplicity of notation, 
we will write
\begin{equation}
\label{eq:omega arr}
\Omega^i_r(\A):=\Omega^i_r(X(\A)), 
\end{equation}
and view this set as lying in the Grassmannian 
$\Grass_r(\A):=\Grass_r(H^1(X(\A),\Q))$ of $r$-planes 
in a rational vector space of dimension $n=\abs{\A}$, 
with a fixed basis given by the meridians around the 
hyperplanes.  Similarly, we will write
\begin{equation}
\label{eq:sigma arr}
\Sigma^i(\A):=\Sigma^i(X(\A),\Z), 
\end{equation}
and view this set as an open subset inside the $(n-1)$-dimensional 
sphere $S(\A)=S(H^1(X(\A),\R))$.

As noted in \cite{PS-plms, Su-aspm}, it follows from work 
of Arapura \cite{Ar} and Esnault, Schechtman and 
Viehweg \cite{ESV} that every arrangement complement 
is locally straight.  In view of Theorem \ref{thm:df straight}\eqref{s1} 
and Corollary \ref{cor:bns straight}, then, we have the 
following corollaries.

\begin{corollary}[\cite{Su-aspm}]
\label{cor:dfarr bound}
For all $i\ge 1$ and $r\ge 1$, 
\begin{equation}
\label{eq:dfarr bound}
\Omega^i_r(\A) \subseteq \Grass_r(\A)\setminus  
\sigma_r(\RR^i(X(\A),\Q)).
\end{equation}
\end{corollary}

\begin{corollary}[\cite{PS-plms}]
\label{cor:sigmaarr bound}
For all $i\ge 1$, 
\begin{equation}
\label{eq:sigmaarr bound}
\Sigma^i(\A) \subseteq S(\A) \setminus  
S(\RR^i(X(\A),\R)).
\end{equation}
\end{corollary}

Since the resonance varieties of $\A$ depend only 
on its intersection lattice, these upper bounds 
for the $\Omega$- and $\Sigma$-invariants are 
combinatorially determined. Furthermore, 
Corollary \ref{cor:bns rv locstraight} yields the 
following. 

\begin{corollary}
\label{cor:dfarr equal}
Suppose $\Sigma^i(\A) = S(\A) \setminus  
S(\RR^i(X(\A),\R))$.  Then 
$\Omega^i_r(\A)= \Grass_r(\A)\setminus  
\sigma_r(\RR^i(X(\A),\Q))$, for all $r\ge 1$.
\end{corollary}

\subsection{Lower bounds for the $\Sigma$-sets}
\label{subsec:kp}
In this context, the following recent result of Kohno and 
Pajitnov \cite{KP} is relevant.    

\begin{theorem}[\cite{KP}]
\label{thm:kp}
Let $\A$ be an arrangement of $n$ hyperplanes, 
and let $\chi=(\chi_1,\dots , \chi_n)$ be a 
vector in $S^{n-1}=S(\A)$, with all components $\chi_j$ strictly 
positive.  Then $H_{i}(X, \widehat{\Z{G}}_{\chi})=0$, 
for all $i< \rk (\A)$. 
\end{theorem}

Let $\chi$ be a positive vector as above.  Making use of 
Theorem \ref{thm:bns novikov}, we infer that 
$-\chi \in  \Sigma^{i}(\A)$, for all $i<\rk(\A)$.  
On the other hand, since $\sum_{j=1}^{n} \chi_j\ne 0$, 
we also have that $-\chi \notin S(\RR^{i}(X(\A),\R))$, 
a fact predicted by Corollary \ref{cor:sigmaarr bound}.

Denote by $S^{-}(\A)=S^{n-1}\cap (\R_{<0})^n$ 
the negative octant in the unit sphere $S(\A)$.
In view of the above discussion, Theorem \ref{thm:bns novikov} 
provides the following lower bound for the BNSR invariants 
of arrangements.  

\begin{corollary}
\label{cor:kp bound}
Let $\A$ be a (central) hyperplane arrangement.  Then 
$S^{-}(\A)\subset \Sigma^i(\A)$, for all $i<\rk(\A)$. 
In particular, $S^{-}(\A)\subset \Sigma^1(\A)$.  
\end{corollary} 

The above lower bound for the BNS invariant of an 
arrangement $\A$ can be improved quite a bit, by 
considering the projectivized arrangement $\overline{\A}$.  
Set $n=\abs{\A}$, and identify the unit sphere 
$S(\overline{\A})=S(H^1(X(\overline{\A}),\R))$ with the 
great sphere $S^{n-2}=\{\chi \in S^{n-1} \mid \sum_{i=1}^n \chi_j =0\}$ 
inside $S^{n-1}=S(\A)$.  Clearly, 
$S^{-}(\A)\subset S(\A)\setminus S(\overline{\A})$. 

\begin{prop}
\label{prop:cone decone}
For any central arrangement $\A$, 
\begin{equation}
\label{eq:sigma1 arr}
\Sigma^1(\A)=S(\A)\setminus 
(S(\overline{\A})\setminus \Sigma^{1}(\overline{\A})). 
\end{equation}
In particular, $S(\A)\setminus S(\overline{\A}) \subseteq \Sigma^1(\A)$. 
\end{prop}

\begin{proof}
By the product formula \eqref{eq:sigmaprod2} for the 
BNS invariants, we have that 
$\Sigma^1(\A)^{\compl}=\Sigma^1(\overline{\A})^{\compl}$. 
The desired conclusion follows.
\end{proof}

\subsection{Discussion and examples}
\label{subsec:disc arr}
In simple situations, the Dwyer--Fried and Bieri--Neumann--Strebel 
invariants of an arrangement can be computed explicitly, and the 
answers agree with those predicted by the upper bounds from 
\S\ref{subsec:omsig arr}.  

\begin{example}
\label{ex:pencil}
Let $\A$ be a pencil of $n\ge 3$ lines through the 
origin of $\C^2$.  Then $X(\A)$ is diffeomorphic to 
$\C^{\times} \times (\C \setminus \set{n-1 \text{ points}})$, 
which in turn is homotopy equivalent to the toric complex $T_L$, 
where $L=K_{1,n-1}$ is the bipartite graph obtained by 
coning a discrete set of $n-1$ points.  Thus, the $\Omega$- 
and $\Sigma$-invariants of $\A$ can be computed using 
the formulas from \S\ref{sect:toric}.   For instance, 
$\RR^1(\A)=\C^{n-1}\subset \C^{n}$.  Therefore,  
$\Omega^1_1(\A)=\QP^{n-1} \setminus \QP^{n-2}$ and 
$\Omega^1_2(\A)=\emptyset$.  Moreover,  
$\Sigma^1(\A)=S^{n-1} \setminus S^{n-2}$, 
which is the same as the lower bound from 
Proposition \ref{prop:cone decone}. 
\end{example}

For arbitrary arrangements, the computation of the 
$\Omega$- and $\Sigma$-invariants is far from being 
done, even in degree $i=1$. 
A more detailed analysis of the $\Omega$-invariants of 
arrangements is given in \cite{Su-aspm, Su11}.  Here is a 
sample result. 

\begin{prop}[\cite{Su-aspm}]
\label{prop:df 12}
Let $\A$ be an arrangement of $n$ lines in $\C^{2}$. 
Suppose $\A$ has $1$ or $2$ lines which contain all 
the intersection points of multiplicity $3$ and higher.  Then 
$\Omega^1_r(\A) =\Grass_r(\Q^n)\setminus  
\sigma_r(\RR^1(X(\A),\Q))$, 
for all $r\ge 1$.
\end{prop}

The reason is that, by a result of Nazir and Raza, the 
first characteristic variety of such an arrangement 
has no translated components, and so $X(\A)$ is $1$-straight. 
In general, though, translated tori in the characteristic variety 
may affect both the $\Omega$-sets and the $\Sigma$-sets 
of the arrangement.

\begin{example}
\label{ex:deleted B3}
Let $\A$ be the deleted $\operatorname{B}_3$ 
arrangement, with defining polynomial 
$Q(\A)=z_0z_1(z_0^2-z_1^2)(z_0^2-z_2^2)(z_1^2-z_2^2)$. 
The  jump loci of this arrangement were computed 
in \cite{Su02}.  The resonance variety 
$\RR^1(\A)\subset \C^8$ contains $7$ 
local components, corresponding to $6$ triple points and one 
quadruple point, and $5$ other components, corresponding 
to braid sub-arrangements.  In particular, $\codim \RR^1(\A)=5$.  
In addition to the $12$ subtori arising from the subspaces  
in $\RR^1(\A)$, the characteristic variety $\VV^1(\A)\subset 
(\C^{\times})^8$ also contains a component of the form 
$\rho\cdot T$, where $T$ is a $1$-dimensional algebraic 
subtorus, and $\rho$ is a root of unity of order $2$.  

Of course, the complement of $\A$ is a formal, 
smooth, quasi-projective variety.  From Theorem \ref{thm:tt}, 
we deduce that the Dwyer--Fried set $\Omega^1_2(\A)$ 
is strictly contained in $\sigma_2(\RR^1(X(\A),\Q))^{\compl}$. 
Using Corollary \ref{cor:dfarr equal}, we conclude that the BNS set 
$\Sigma^1(\A)$ is strictly contained in $S(\RR^1(X(\A),\R))^{\compl}$.  
\end{example}

This example answers in the negative Question 9.18(ii) 
from \cite{Su-conm}.  It would be interesting to compute 
explicitly the $\Omega$-invariants and $\Sigma$-invariants 
of wider classes of arrangements, and see whether these 
invariants depend only on the intersection lattice, or also 
on other, more subtle data. 

\section{Acknowledgements}
\label{sec:ack}

A preliminary version of this paper was presented at the 
Centro di Ricerca Matematica Ennio De Giorgi in Pisa, 
Italy, in May--June, 2010.  I  wish to thank the organizers of 
the Intensive Research Period on {\em Configuration Spaces: 
Geometry, Combinatorics and Topology}\/ for their friendly 
hospitality.

Most of the work was done during the author's visit at the 
Universit\'{e} de Caen, France in June, 2011. Likewise, 
I wish to thank the Laboratoire de Math\'ematiques 
Nicolas Oresme for its support and hospitality. 

Finally, I wish to thank the referee for helpful comments 
and suggestions.

\newcommand{\arxiv}[1]
{\texttt{\href{http://arxiv.org/abs/#1}{arXiv:#1}}}
\newcommand{\arx}[1]
{\texttt{\href{http://arxiv.org/abs/#1}{arXiv:}}
\texttt{\href{http://arxiv.org/abs/#1}{#1}}}
\newcommand{\doi}[1]
{\texttt{\href{http://dx.doi.org/#1}{doi:#1}}}
\renewcommand{\MR}[1]
{\href{http://www.ams.org/mathscinet-getitem?mr=#1}{MR#1}}

\end{document}